\documentclass[a4paper]{article}
\usepackage {amssymb,latexsym}
\usepackage {amsmath}
\usepackage{amsmath,amssymb}
\usepackage[all]{xy}
\usepackage{amscd}
\newtheorem{thm}{Theorem}[section]
\newtheorem{prop}{Proposition}[section]
\newtheorem{cor}{Corollary}[section]
\newtheorem{lem}{Lemma}[section]
\newtheorem{de}{Definition}[section]

\newtheorem{rem}{Remark}[section]
\newenvironment{proof}{
                      \noindent{Proof.}}
                                      {\hfill {$\mathbf \Box$}\medskip}
\newcommand{\got}[1]{\mathfrak{#1}}

\newcommand{\K}{\mathbb{K}}
\newcommand{\N}{\mathbb{N}}

\newcommand{\C}{\mathbb{C}}

\newcommand{\Q}{\mathbb{Q}}

\newcommand{\HH}{\mathrm{H}}
\newcommand{\Hom}{\mathrm{Hom}}

\newcommand{\Def}{\mathrm{Def}}
\newcommand{\J}{\mathrm{J}}
\newcommand{\I}{\mathrm{I}}

\newcommand{\Der}{\mathrm{Der}}

\newcommand{\D}{\mathrm{D}}

\newcommand{\m}{\mathfrak{m}}
\newcommand{\n}{\mathfrak{n}}
\newcommand{\Gl}{\mathrm{Gl}}
\newcommand{\G}{\mathrm{G}}
\newcommand{\M}{\mathrm{M}}
\newcommand{\Aut}{\mathrm{Aut}}

\newcommand{\gl}{\mathfrak{gl}}
\newcommand{\g}{\mathfrak{g}}
\newcommand{\V}{\mathrm{V}}
\newcommand{\A}{\mathrm{A}}

\renewcommand{\L}{\got{L}}
\renewcommand{\O}{\mathcal{O}}
\makeatletter
\@addtoreset{equation}{section}

\makeatother
\title{Versal Deformations and Versality in Central Extensions of Jacobi's Schemes}
\author{Roger Carles\\UMR 6086 du CNRS, Laboratoire de Math\'{e}matiques et Applications \\Universit\'e de Poitiers, F-8692 Futuroscope Chasseneuil France\\{carles@math.univ-poitiers.fr}
\\ Toukaiddine Petit \footnote{\tt Supported by the EC project Liegrits MCRTN 505078.}
\\Departement Wiskunde en Informatica, Universiteit Antwerpen\\
B-2020 (Belgium)\\{toukaiddine.petit@ua.ac.be}.
}    
\date{}
\begin{document}
\maketitle
\textbf{Abstract.} Let $\L_m$ be the scheme of the laws defined by the Jacobi's identities on $\K^m$ with $\K$ a field. A deformation of $\g\in\L_m$, parametrized by a local $\K$-algebra $\A$, is a local $\K$-algebra morphism from the local ring of $\L_m$ at $\phi_m$ to $\A$. The problem to classify all the deformation equivalence classes of a Lie algebra with given base is solved by "versal" deformations. First, we give an algorithm for computing versal deformations. 
Second, we prove there is a bijection between the deformation equivalence classes of an algebraic Lie algebra $\phi_m=\mathrm{R}\ltimes\varphi_n$ in $\L_m$ and its nilpotent radical $\varphi_n$ in the $\mathrm{R}$-invariant scheme $\L_n^{\mathrm{R}}$ with reductive part $\mathrm{R}$, under some conditions. So the versal deformations of $\phi_m$ in $\L_m$ is deduced to those of $\varphi_n$ in $\L_n^{\mathrm{R}}$, which is a more simple problem. Third, we study versality in central extensions of Lie algebras. 
Finally, we calculate versal deformations of some Lie algebras.
\section*{Introduction}
Formal deformation theory of associative algebras and Lie algebras over an algebraically closed field $\K$ of characteristic zero has been first studied by Gerstenhaber \cite{G}, Nijenhuis and Richardson \cite{NR}. For more information about this topic see \cite{B,C9,G,Ko,NR}. Schlessinger generalized deformation theory from the base $\K[[t]]$ to a commutative local $\K$-algebra, \cite{S}. A Lie algebra of dimension $m$ is viewed as a point $\phi_0$ of the scheme $\L_m$ defined by the Jacobi's identities and antisymmetry. Prolonging these ideas, we consider in \cite{C3}, a deformation of $\phi_0$, parametrized by a local $\K$-algebra $\A$, with maximal ideal $\m$, as a local $\K$-algebra morphism $\O\rightarrow\A$, with $\O$ the local ring of $\L_m$ at $\phi_0$. There is a subgroup $\G_m(\A)$ of the linear group $\Gl_m(\A)$ acting on the set of deformations $\Def(\phi_0,\A)$. An interesting and largely open problem is to classify all the deformation equivalence classes with fixed base of a Lie algebra. Considering deformations with any base permits us to simplify this problem. In fact, this is equivalent to look for deformations called versal $g:\O\rightarrow\mathrm{R}$ with base $\mathrm{R}$ satisfying that for  
any deformation $h:\O\rightarrow\mathrm{A}$, with base $\A$, there is a local morphism $\overline{h}:\mathrm{R}\rightarrow\mathrm{A}$ (a base change) such that $\overline{h}\circ g$ is equivalent to $h$ and the Zariski's tangent space of $\mathrm{R}$ is minimal for this property.

The aim of this paper is to provide an algorithm for computing versal deformations, a theorem called reduction theorem which reduces their computation for a class of algebraic Lie algebras, and to study versality in central extensions of subschemes of $\L_m.$

In Section $1$, we give classical material of deformations in this new definition where the properties depend on the $\K$-algebra $\A$. We introduce the schemes $\L_n^\mathrm{D}$ consisting laws which are invariant under $\D\subset\gl_n(\K)$, and $\Def(\varphi_0,\A)^{\mathrm{D}}$ the set of deformations of $\varphi_0\in\L_n^\mathrm{D}(\K)$ defined by the local ring $\O^{\mathrm{D}}$ of $\L_n^{\D}$ at $\varphi_0$. We assume that $\mathrm{D}$ is completely reducible in $\gl_n(\K)$. The subgroup $\G_n(\A)^{\D}$ of $\G_n(\A)$ consisting elements which commute with $\mathrm{D}$ acts on $\Def(\varphi_0,\A)^{\mathrm{D}}$. 

In Section $2$, we provide an algorithm for computing versal deformations of $\phi_0$. This method is based on the fixing some parameters indexed by a set $\mathcal{A}$ called admissible set at $\phi_0$. The bases of versal deformations of $\phi_0$ are the local rings of slices  $\L_{m,\phi_0}^{\mathcal{A}}$ which are transversal subschemes to the orbit of $\phi_0$ in $\L_m(\K)$ under $\Gl_m(\K)$. 

In section $3$, we give the reduction theorem : Let $\g=\mathrm{R}\ltimes\n$ be an algebraic Lie algebra with reductive part $\mathrm{R}$ and nilpotent radical $\n$ of dimension $n$. For any Noetherian complete local $\K$-algebra $\A$, there is a bijection between $\Def(\g,\A)/\G_m(\A)$ and $\Def(\n,\A)^{\mathrm{R}}/\G_n(\A)^{\mathrm{R}}$; and the local ring of any slice of $\L_m$ at $\g$ is isomorphic to that of any slice of $\L_n^\mathrm{D}$ at $\n$, for an important class of algebraic Lie algebras. So the versal deformations of $\g$ in $\L_m$ is deduced to those of $\n$ in $\L_n^{\mathrm{R}}$, which is a more simple problem. Also we provide a new criterion of formal rigidity of $\phi_0$ saying that if the Krull's dimension of the completion local ring  of any slice at $\phi_0$ is zero, then $\phi_0$ is formal rigid. 

Under some hypotheses, we could limit the local study of $\g=\mathrm{R}\ltimes\n$ in $\L_m$ to that of $\n$ in $\L^\mathrm{R}_n$ according to the reduction theorem. Moreover, if the weights of the center $\mathrm{T}$ of $\mathrm{R}$ were different from zero, the scheme $\L_n^{\mathrm{R}}$ would be consisted of nilpotent laws. Passing from $\L_m$ to $\L^\mathrm{R}_n$ has the advantage to use methods which are specific for nilpotent Lie algebras. In particular, we can use the construction method of nilpotent Lie algebras based on successive central extensions.
Then we build in section $4$ some sequences of schemes $(\L_{n}^{\mathrm{T}})_n$ for $n\geq n_0$ consisting of nilpotent laws whose  
their versal deformations are made  by successive central extensions from that of $\varphi_{n_0}\in\L_{n_0}^{\mathrm{T}}$, with $n_0$ the dimension where the maximal torus $\mathrm{T}$ appears. This leads to a continuous
family concept. 

In Section $5$, we calculate versal deformations of some Lie algebras.		
\section{Generality}
$\textbf{1.1}$ In this work, we consider an algebraically closed field $\K$ of characteristic $0$. 
Let $(e_i)$ be a basis of $\V:=\K^m$, $\A$ a commutative associative $\K$-algebra with unity $1=1_\A$, and $\L_m\left(\A\right)$ the set of all Lie $\A$-algebra multiplications on $\A\otimes_\K\V$. An element $\phi$ of $\L_m\left(\A\right)$ is defined by its structure constants $\phi^k_{ij}\in\A$: $\phi\left(e_i,e_j\right)=\sum_{k=1}^{m}\phi^k_{ij}e_k$. Then the set $\L_m\left(\A\right)$ can be identified with the set of $(\phi^k_{ij})\in\A^{\frac{m(m+1)}{2}}$\quad such that
\begin{equation}\label{n1}
\phi^k_{ij}+\phi^k_{ji}=0,\quad J_{ijk}^l=\sum_{s=1}^{p}\left(\phi^s_{ij}\cdot\phi^l_{sk}+\phi^s_{jk}\cdot\phi^l_{si}+\phi^s_{ki}\cdot \phi^l_{sj}\right)=0.
\end{equation}
It is well-known that $\L_m$ is a functor from the category of commutative associative $\K$-algebras to the category of sets. Denote by $\J_m$ the ideal of the polynomial ring $\mathrm{P}_m:=\K\left[X^\alpha:\alpha\in\mathcal{I}\right]$ generated by the relations (\ref{n1}), and $\I_m$ the quotient algebra $\I_m:=\mathrm{P}_m/\J_m$, where $\mathcal{I}=\left\{(^k_{ij}):1\leq i<j\leq m,1\leq k\leq m\right\}$ is the set of all multi-indices. For any $\K$-algebras $\A$, $\mathrm{B}$, let us denote by $\mathrm{Hom}_{\mathrm{Alg}}(\A,\mathrm{B})$ the set of $\K$ algebra homomorphisms from $\A$ to $\mathrm{B}$. Let $x^\alpha$ be the class of $X^\alpha$ in $\I_m$. The $\K$-functor $\L_m$ is representable by $\I_m$ i.e there is a canonical bijection  $\L_m(\A)\rightarrow\mathrm{Hom}_{\mathrm{Alg}}(\I_m,\A),\,\phi\rightarrow f_\phi$ defined by $f_\phi:\I_m\rightarrow\A$ with $f_\phi(x^\alpha)=\phi^\alpha$, for all commutative associative $\A$. Giving a point $\phi\in\L_m(\A)$ is equivalent to giving an algebra morphism from $\I_m$ to $\A$. The $\K$-functor $\L_m$ is an affine algebraic $\K$-scheme with algebra $\I_m$, and we have $\L_m\simeq\mathrm{Spec}\left(\I_m\right)$, \cite{R}.  The set of rational points of $\L_m$ is identified with the set $\L_m\left(\K\right)$ of all Lie $\K$-algebra multiplications on $\V$. \\
$\textbf{1.2}$ Let $\A$ be a local $\K$-algebra with maximal ideal $\m=\m\left(\A\right)$, residue field $\K=\A/\m$, augmentation map $\mathrm{pr}:\A\rightarrow\K$, and let $\widehat{\A}$ denote its  $\m$-adic completion $\K$-algebra. \\
A deformation of a point $\phi_0$ of $\L_m\left(\K\right)$ with base $\A$ is a point $\phi$ of $\L_m\left(\A\right)$ such that $\phi_0=\L_m\left(\mathrm{pr}\right)\left(\phi\right)$. For all fixed point $\phi_0$, we obtain a functor $\Def\left(\phi_0,-\right)$ from the category of commutative associative local $\K$-algebras to the category of the sets. We extend $\Def\left(\phi_0,-\right)$ to $\widehat{\A}$ by taking inverse limits. \\If $\O=\O_{\phi_0}$ is the local ring of the scheme $\L_m$ at $\phi_0$, then it is equal to the localization ring of $\I_m$ by the maximal ideal $\m_0:=\mathrm{Ker}(f_0)$, with $f_0:=f_{\phi_0}$.\\
Giving a deformation $\phi$ of $\phi_0$ with base $\A$ is equivalent to giving a $\K$-algebra morphism, $f=f_{\phi}:\I_m\rightarrow\A$ such that $f_0=\mathrm{pr}\circ f$. This is equivalent to $f$ sends $\m_0$ to $\m$. Then $f$ induces a local morphism from $\O$ to $A$. We deduce that giving a deformation $\phi$ of $\phi_0$, is equivalent to give a local $\K$-algebra morphism $f:\O\rightarrow\A$. Then the set of deformation $\mathrm{Def}(\phi_0,\A)$ can be identified with the set of local $\K$-algebra homomorphisms $\mathrm{Hom}_{\mathrm{alg}}(\O,\A)$. The case where $\A=\K[[t]]$ corresponds to the Gerstenhaber's deformations  \cite{G}.\\
The deformation $\mathrm{id}:\O\rightarrow\O$ defined by $\mathrm{id}(x^\alpha)=x^\alpha$, is called deformation identity or canonical deformation. It is an initial universal object in the category of deformations at the point $\phi_0$. Any deformation of $\phi_0$ with  base $\A$ may be deduced from the identity deformation by a base change.
A deformation $\phi\in\Def\left(\phi_0,\A\right)$ is defined by its structure constants, $\phi^\alpha=f\left(x^\alpha\right)$. The smallest subalgebra of $\A$ containing them is equal to $f\left(\O\right)$. It is a Noetherian local $\K$-algebra since it is isomorphic to 
$\O/\ker\left(f\right)$. Next we can assume that $\A$ is Noetherian. Let $\got{R}$ denote the category of Noetherian local $\K$-algebras, $\widehat{\got{R}}$ that of Noetherian complete local $\K$-algebras.\\
$\textbf{1.3}$ Let $f\in\Hom_{\mathrm{alg}}\left(\O,\A\right)$ be a deformation with $\A\in\widehat{\got{R}}$, then, for each $p\geq 0$, its reduction modulo $\m^{p+1}$ is a deformation with base $\A/\m^{p+1}$. We obtain a map
\begin{equation}\label{e14}
\Hom_{\mathrm{alg}}\left(\O,\A\right)\quad\longrightarrow\quad\underleftarrow{\mathrm{lim}_p}\Hom_{\mathrm{alg}}\left(\O,\A/\m^{p+1}\right)
\end{equation}
It follows from the bijection (\ref{e14}) that for each element $f$ of $\Hom_{\mathrm{loc}}\left(\O,\A\right)$, there is a family $\left(f_p\right)_{p\in\N}$ with $f_p\in\Hom_{\mathrm{loc}}(\O,\A/\m^{p+1})$ such that $f=\underleftarrow{\mathrm{lim}}\left(f_p\right)$ when $p\rightarrow\infty$. The statement that $f$ forms a deformation is equivalent to lift each local morphism $f_p$ to the local morphism $f_{p+1}$. This is equivalent to the commutativity of the following diagram:
\begin{equation}
	\xymatrix{
	\O
	\ar@{->}[r]^{f_{p+1}} \ar@{->}[dr]^{f_{p}}
	& \A/\m^{p+2}
		\ar@{->}[d]^{\pi}\\
	&\A/\m^{p+1}}
	\end{equation}
where $\pi$ is the projection. We say that there is a $(p+1)$-obstruction if it is impossible to find such a solution.\\
$\textbf{1.4}$ Consider the Chevalley-Eilenberg's complex $(\mathrm{C}\left(\g,\g\right),d)$ with adjoint values of a Lie algebra $\g$. The Nijenhuis-Richardson's bracket of a $m$-cochain $f$ by a $q$-cochain $g$ is the commutator $\left[f,g\right]=f\bullet g-(-1)^{(m-1)(q-1)}g\bullet f.$
It may be checked that $d f=(-1)^{m+1}\left[\phi_0,f\right]$.
Denote by $\mathrm{Z}^m\left(\g,\g\right)$, $\mathrm{B}^m\left(\g,\g\right)$ and $\mathrm{H}^m\left(\g,\g\right)$, the space of $m$-cocycles, $m$-coboundaries, and $m$-cohomologies respectively, \cite{NR}.\\
$\textbf{1.5}$ Let $\A\in\widehat{\got{R}}$. If $r$ is the dimension of $\m/\m^2$ over $\K$, we then have $\A=\K\left[\left[T_1,...,T_r\right]\right]/\mathfrak{a}$ with $\mathfrak{a}$ an ideal. 
If $f\in\mathrm{Hom}_{\mathrm{alg}}(\O,\A)$ is a deformation, its structure constants $f\left(x^\alpha\right)=\phi^\alpha$ may be written as formal series in generators $t_1,...,t_r$ of $\m\left(\A\right)$, say $\phi^\alpha=\phi^\alpha\left(t_1,...,t_r\right),$ 
with $(\phi_0)^\alpha=\phi^\alpha\left(0,...,0\right).$ 
This deformation may be written as $(\phi^\alpha)=\phi\left(t\right)=\sum_{\mu\in\N^r}t^\mu\phi_\mu,$ with $\phi_\mu\in\mathrm{C}^2\left(\g,\g\right)$, and $t^\mu=t_1^{\mu_1}...t_r^{\mu_r}$. This writing is not unique if $\mathfrak{a}\neq 0$.\\
$\textbf{1.6}$ Let $\A$ be a local $\K$-algebra and $\Gl_{m}(\A)$ the group of invertible $m$-matrices with coefficient in $\A$.
The augmentation map $\mathrm{pr}:\A\rightarrow\K$ induces a group $\K$-morphism denoted again by ${\mathrm{pr}}$, ${\mathrm{pr}}:\Gl_{m}\left(\A\right)\rightarrow\Gl_{m}\left(\K\right)$.
Its kernel denoted by $\G_{m}\left(\A\right)$ is a subgroup of $\Gl_{m}\left(\A\right)$ equal to $\mathrm{id}+\M_{m}\left(\m\right)$, where $\M_{m}\left(\m\right)$ is the set of matrices with coefficients in $\m$. The group $\Gl_{m}\left(\A\right)$ canonically acts on $\L_{m}\left(\A\right)$ by
\begin{equation}
	s\ast\phi=s\circ\phi\circ(s^{-1}\times s^{-1})
\end{equation}
i.e.
\begin{equation}\label{A}
	(s\ast\phi)^k_{ij}=\sum_{lpq} (s^{-1})_{li}\,(s^{-1})_{pj}\,s_{kq}\,\phi^q_{lp}
\end{equation}
with $s=(s_{ij})\in\Gl_{m}\left(\A\right)$ and $\phi\in\L_{m}\left(\A\right)$. Denote by $\left[\phi\right]$ the $\Gl_{m}\left(\A\right)$-orbit of $\phi$. The action of the group $\G_m(\A)$ on $\mathrm{Hom}_{\mathrm{alg}}(\O,\A)$  is given by 
\begin{equation}\label{B}
	s\ast f_{\phi}:=f_{s\ast\phi}:\O\rightarrow\A, x^k_{ij}\longmapsto (s\ast\phi)^k_{ij}
\end{equation}
with $s\in\G_{m}\left(\A\right)$ and $f_{\phi}\in\mathrm{Hom}_{\mathrm{alg}}(\O,\A)$.\\
Two deformations of $\phi_0$ with base $\A$ are said to be equivalent if they lie on the same orbit under $\G_{m}\left(\A\right)$. Denote by $\overline{\Def}\left(\phi_0,\A\right)$ the set of deformation equivalence classes of $\phi_0$ with base $\A$.\\
A deformation is said to be $\A$-trivial if it is equivalent to the constant deformation $c:\O\rightarrow\A$ defined by $c\left(u\right):=f_{0}\left(u\right)\cdot 1_{\A}$, where $u\in\O$. A point $\phi_0$ of $\L_{m}\left(\K\right)$ is said to be $\A$-rigid if each deformation of $\phi_0$ is $\A$-trivial i.e. $\Def\left(\phi_0,\A\right)$ consists of only one orbit under $\G_{m}\left(\A\right)$.
The notion of $\A$-rigidity depends on the base $\A$. The geometric rigidity which means that the orbit of $\phi_0$ under $\Gl_{m}\left(\K\right)$ is a Zariski's open set in $\L_{m}\left(\K\right)$, corresponds to the rigidity with base $\A:=\K[[t]]$, \cite{NR}. 
The Zariski's tangent space to $\L_m(\K)$ (resp. the $\Gl_m(\K)$-orbit $\left[\phi_0\right]$ of $\phi_0$) at $\phi_0$ identifies with $\mathrm{Z}^2(\phi_0,\phi_0)$ (resp. $\mathrm{B}^2(\phi_0,\phi_0)$). We have
$\mathrm{H}^2(\phi_0,\phi_0)=0$ if and only if the $\Gl_m(\K)$-orbit of $\phi_0$ is a Zariski's open set in $\L_{m}\left(\K\right)$, and the scheme $\L_{m}$ is reduced at point $\phi_0$.\\
We suppose that $\A$ is complete, then each element $s$ of $\G_{m}\left(\A\right)$ may be written as $\sum_{\mu\in\N^r}t^\mu s_\mu$, with $s_0=\mathrm{id}$ and $s_\mu\in\M_{m}\left(\K\right)$. Let $k$ be the smallest length nonzero of $\mu$ such that $\sum_{|\mu|=k}t^\mu s_\mu$, is nonzero. The action of $s$ on a deformation $\phi=\sum_{\mu\in\N^r}t^\mu\phi_\mu$ of $\phi_0$ with base $\A$, is given by
\begin{equation}\label{e4}
	s\ast\phi\equiv\sum_{|\mu|<k}t^\mu\phi_\mu+\sum_{|\mu|=k}t^\mu\left(\phi_\mu-d_{\phi_0} s_\mu\right)\mod\m^{k+1}
\end{equation}
$\textbf{1.7}$ Let $\D$ be a subset of $\mathrm{M}_m(\K)$. Let $\L_m^{\D}\left(\K\right)$ denote the set of laws $\phi$ of $\L_m\left(\K\right)$ which admit $\D$ as derivations i.e. $\delta\cdot\phi=0$, for all $\delta=(\delta^i_j)\in\D$.  
Let $\Delta_m$ be the ideal of $\mathrm{P}_m$ generated by the polynomials
\begin{equation}\label{E1.14}
\sum^p_{l=1}\left(X^l_{ij}\delta^k_l-\delta^l_iX^k_{lj}-\delta^l_jX^k_{il}\right),\left(i,j,k\in\left\{1,...,m\right\},i<j\right)
\end{equation}
for all $\delta\in\D$. Denote by $\J^{\D}_m$ the sum of $\J_m$ and $\Delta_m$, and $\I_m^{\D}$ the quotient algebra $\mathrm{P}_m/\J_m^{\D}$. 
Let $\L_m^{\D}$ denote the subscheme of $\L_m$ which is canonically isomorphic to $\mathrm{Spec}\left(\I_m^{\D}\right)$. Let $\phi_0$ be a point of $\L_m^{\D}(\K)$, $\A$ a local $\K$-algebra and let $\Def^{\D}(\phi_0,A)$ denote the set of deformations of $\phi_0$ with base $\A$ in $\L_m^{\D}(\A)$.
The quotient algebra $\K$-morphism from $\I_m$ to $\I_{m}^{\D}$ induces a local algebra $\K$-morphism $\eta:\O_{\phi_0}\rightarrow\O_{\phi_0}^{\D}$ and a  scheme embedding from $\L_{m}^{\D}$ onto $\L_m$, denoted by $\mathfrak{I}$.
Also it induces a map denoted again by $\mathfrak{I}$ and defined by
$$\Hom_{\mathrm{Alg}}(\O_{\phi_0}^\D,\A)\longrightarrow\Hom_{\mathrm{Alg}}(\O_{\phi_0},\A),\quad f\longmapsto \mathfrak{I}(f)=f\circ\eta.$$
If $\left[\phi_0\right]$ is an orbit of a law $\phi_0$ in $\L_m\left(\K\right)$, we denote by $\left[\phi_0\right]^{\D}=\left[\phi_0\right]\cap\L_{m}^{\D}\left(\K\right)$ its trace in $\L_{m}^{\D}\left(\K\right)$. 
Assume that $\D$ is a completely reducible Lie subalgebra of $\mathrm{gl}(\K^n)$, i.e. its natural representation on $\V$ is semisimple. The normalizer subgroup $\HH$ of $\D$ in $\Gl\left(\K^n\right)$ stabilizes $\L_m^\D\left(\K\right)$ and the trace orbits. Since the adjoint representation of $\D$ is semisimple, hence the Lie algebra of $\mathrm{H}$ is given by $\D+\mathrm{gl}(\K^n)^{\mathrm{D}}$. 
If $\mathrm{H}_0$ is the identity component of $\mathrm{H}$, it is proved in \cite{C7} that \\
	$\textbf{a)}$ the Zariski's tangent space to $\L_{m}^{\D}\left(\K\right)$ at $\phi_0$ is equal to $\mathrm{Z}^2\left(\g,\g\right)^\D$.\\
	$\textbf{b)}$ $[\phi_0]^{\D}$ is a finite union of orbits under $\mathrm{H}_0$ in $\L_m^\D(\K)$.\\
	$\textbf{c)}$ The Zariski's tangent space to $[\phi_0]^\D$ at $\phi_0$, is equal to that of its orbit under $\mathrm{H}_0$ at $\phi_0$. It is also equal to $\mathrm{B}^2\left(\phi_0,\phi_0\right)^{\D}$.\\
	$\textbf{d)}$ The orbits under $\mathrm{H}_0$ in $\L_m^\D(\K)$ are the same that under the identity component $\Gl\left(\K^n\right)^{\D}_0$ of the group  $\Gl\left(\K^n\right)^{\D}$ consisting of the elements of $\Gl\left(\K^n\right)$ commuting with $\D$.\\
	$\textbf{e)}$ For $\K=\C$, we have $\mathrm{H}^2\left(\phi_0,\phi_0\right)^{\D}=0$ if and only if the $\mathrm{H}_0$-orbit of $\phi_0$ is a Zariski's open set in $\L_{m}^{\D}\left(\K\right)$ and the scheme $\L_{m}^{\D}$ is reduced at $\phi_0$.\\
$\textbf{1.8}$ 
Let $\mathcal{A}$ be a subset of $\mathcal{I}$ and let $\phi_0$ be a point of $\L_m\left(\K\right)$. Set 
\begin{equation}\label{E1.15}
	\J^\mathcal{A}_{m,\phi_0}:=\J_m+\left\langle  X^{\alpha}-(\phi_0)^{\alpha}:\alpha\in\mathcal{A}\right\rangle,\ \I^\mathcal{A}_{m,\phi_0}:=\mathrm{P}_m/\J^\mathcal{A}_{m,\phi_0},
\end{equation}
we define a subscheme $\L^\mathcal{A}_{m,\phi_0}$ of $\L_m$ as 
\begin{equation}\label{e1.8}
\L^\mathcal{A}_{m,\phi_0}(\A):=\left\{\phi\in\L_m(\A)|\phi^\alpha=(\phi_0)^\alpha,\alpha\in\mathcal{A}\right\}
\end{equation}
$\mathrm{\textbf{Semi-direct product}}$. Let $n\in\N$ such that $n\leq m$. Set 
\begin{equation}\label{E1.17}
\mathcal{I}'=\left\{\left(^k_{ij}\right)\in\mathcal{I}:i,j,k\leq n\right\}, \mathrm{P'}_n=\K[X^\alpha:\alpha\in\mathcal{I'}]
\end{equation}
We assume that $\phi_0$ is the law of a semi-direct product Lie algebra $\g$ with a reductive part $\mathrm{R}=\left\langle e_{n+1},..,e_m\right\rangle$ and the nilpotent part $\mathfrak{n}=\left\langle e_1,...,e_n\right\rangle$. The adjoint action of $\mathrm{R}$ on $\n$ defines derivations $\delta_i$ of the restriction $\varphi_0$ of $\phi_0$ on $\n$ by
\begin{equation}\label{E1.18}
	\delta_i e_j=\phi_0\left(e_{n+i},e_j\right)=\sum_{k=1}^nC^k_{n+i,j} e_k,\,i\leq r,j\leq n\,\,\mathrm{if}\,r=m-n.
\end{equation}
The elements $(\delta_i)_{i\leq r}$ generate the reductive Lie algebra 
$\mathrm{D}=\mathrm{ad}_{\n}\mathrm{R}$.
\begin{prop}\label{m12} If $	\L^{\mathcal{I}-\mathcal{I}'}_{m,\phi_0}$ is the subscheme of $\L_m$ defined by $\mathcal{I}-\mathcal{I}'$ as above, then there is a scheme isomorphism $\L^{\mathcal{I}-\mathcal{I}'}_{m,\phi_0}\simeq\L^\mathrm{R}_{n}.$
\end{prop}
\begin{proof} The $\K$-algebra $\I^{\mathcal{I}-\mathcal{I}'}_{m,\phi_0}$ which defines the scheme $\L^{\mathcal{I}-\mathcal{I}'}_{m,\phi_0}$ is isomorphic to $\mathrm{P'}_n/\J^\mathcal{I'}_{m,\phi_0}$ where $\J^\mathcal{I'}_{m,\phi_0}$ is the ideal defined by 
$\J_m/\left\langle  X^{\alpha}-(\phi_0)^{\alpha}:\alpha\in\mathcal{I}-\mathcal{I}'\right\rangle\cap\J_m.$
It suffices to set $X^\alpha=(\phi_0)^\alpha$ for all $\alpha\in\mathcal{I}-\mathcal{I}'$ in the Jacobi's polynomials $J^d_{abc}$. For $a,b,c\leq n$ we obtain the Jacobi's polynomials in function of coordinates $(X^\alpha)_{\alpha\in\mathcal{I'}}$ since $\n$ remains a Lie subalgebra. If $a=n+i$, $b$ and $c$ $\leq n$, then the expressions $\delta_i\varphi(e_b,e_c)-\varphi(e_b,\delta_i e_c)-\varphi(\delta_i e_b,e_c)$ generate the ideal $\Delta_m$ defined by (\ref{E1.14}). The other Jacobi's polynomials give null constants. We obtain \\$\I^{\mathcal{I}-\mathcal{I}'}_{m,\phi_0}\cong\mathrm{P'}_n/(\J_m\cap\left\langle X^\alpha, \alpha\in\mathcal{I}'\right\rangle+\Delta_m)=\I_{n,\varphi_0}^{\mathrm{R}}.$
\end{proof}\\
\textbf{Notation}. We shall denote by $\L_n^{\mathrm{R}}$ instead of $\L_n^\mathrm{D}$ for $\mathrm{D}=\mathrm{ad}_\n\mathrm{R}$.
The $\K$-epimorphism algebra from $\I_m$ to $ \I^\mathrm{R}_{n}$ induces a scheme embedding from $\L^\mathrm{R}_{n}$ to $\L_m$, denoted by $\mathfrak{I}$. Let $\A$ be a $\K$-algebra, the morphism 
$\mathfrak{I}:\L^\mathrm{R}_{n}(\A)\longrightarrow\L_m(\A),(\varphi^\alpha)\mapsto(\phi^\alpha)$, defined by
$\phi^\alpha=\varphi^\alpha\quad\mathrm{if}\quad\alpha\in\mathcal{I'}$ and $\phi^\alpha=\left(\phi_0\right)^\alpha\quad\mathrm{if}\quad \alpha\in\mathcal{I}-\mathcal{I}'.$\\
Denote by $\O^{\mathrm{R}}_{\varphi_0}$ the local ring of $	\L^\mathrm{R}_{n}$ at $\varphi_0$, the morphism $\mathfrak{I}$ induces a local epimorphism  $\eta$
\begin{equation}
	\eta:\O_{\phi_0}\rightarrow\O_{\varphi_0}^\mathrm{R}.
\end{equation}
Denote again by $\mathfrak{I}$ the map defined by
$$\Hom_{\mathrm{Alg}}(\O_{\phi_0}^\mathrm{R},\A)\longrightarrow\Hom_{\mathrm{Alg}}(\O_{\phi_0},\A),\quad f\longmapsto \mathfrak{I}(f)=f\circ\eta.$$
\section{Construction of Versal Deformations}
In this section, we provide an algorithm for computing versal deformations of a Lie algebra by fixing some parameters.  The idea to fix some structure constants of a Lie algebra was used in \cite{C9} for parameterizing a deformation. Fuchs and Fialowski provide in \cite{FF} an algorithm of computation the completion of a base of miniversal deformations based on calculations a basis of $\mathrm{H}^i(\g,\g)$ with $i=2,3$ and their Massey products. 
Their method is valid if the dimension of $\mathrm{H}^2(\g,\g)$ is finite, but treats only the miniversal case. 
\subsection{Parametrization of deformations}
Let $\A\in\widehat{\got{R}}$ and a basis $(e_\alpha)_{\alpha\in\mathcal{I}}$ of $\mathrm{C}^2(\V,\V)$. A deformation $\phi=\sum_{\alpha\in\mathcal{I}}\phi^\alpha e_\alpha$ of $\phi_0$ (or $f:\O\rightarrow\A)$ can be written as  $f(x^\alpha)=\phi^{\alpha}=(\phi_0)^\alpha+\xi^\alpha$ with $\xi^\alpha\in\m(\A)$. A family of parameters of $\phi$ or $f$ is a family  $(t_i)_{1\leq i\leq r}$ of the maximal ideal $\m=\m(f(\O))$ such that the classes $\overline{t}_i$ modulo $\m^2$ form a $\K$-basis of $\overline{\m}:=\m/\m^2$. The number of parameters of the deformation is bounded by $\dim_\K\mathrm{Z}^2(\phi_0,\phi_0)$.\\
If $\mathrm{H}^k$ is a complement of $\mathrm{B}^{k}(\g,\g)$ in $\mathrm{Z}^{k}(\g,\g)$ and  $\mathrm{W}^k$ a complement of $\mathrm{Z}^k(\g,\g)$ in $\mathrm{C}^k(\g,\g)$, then $\mathrm{H}^k$ is canonically isomorphic to $\mathrm{H}^k(\g,\g)$, and $d$ induces an isomorphism from $\mathrm{W}^k$ onto $\mathrm{B}^{k+1}(\g,\g)$. Thus we obtain a decomposition
\begin{equation}\label{e3.1}
	\mathrm{C}^k(\g,\g)=\mathrm{Z}^k(\g,\g)\oplus\mathrm{W}^k=\mathrm{B}^k(\g,\g)\oplus\mathrm{H}^k\oplus\mathrm{W}^k,
\end{equation}
called decomposition of Hodge associated with $\g$. The law $\phi=\phi_0+\xi$ is a deformation of $\phi_0$ if and only if $\xi$ is a solution of the Maurer-Cartan's equation
\begin{equation}\label{e10}
\mathrm{MC}_m\left(\m\right)=\left\{\xi\in\m\otimes\mathrm{C}^2\left(\g,\g\right):d\xi-\frac{1}{2}\left[\xi,\xi\right]=0\right\},
\end{equation}
where the differential $d$ extends to $\A\otimes\mathrm{C}\left(\g,\g\right)$ by $\mathrm{id}_{\A}\otimes d$ and denote it again by $d$. We have
\begin{equation}\label{e3.3}
d(\xi)-\frac{1}{2}\left[\xi,\xi\right]=d(\xi)-\frac{1}{2}\left[\xi,\xi\right]_{\mathrm{B}^3}-\frac{1}{2}\left[\xi,\xi\right]_{\mathrm{H}^3}-\frac{1}{2}\left[\xi,\xi\right]_{\mathrm{W}^3},
\end{equation}
where $\left[\xi,\xi\right]_{\mathrm{B}^3},\left[\xi,\xi\right]_{\mathrm{H}^3},\left[\xi,\xi\right]_{\mathrm{W}^3}$ are the projections on $\m\otimes\mathrm{B}^3(\g,\g)$, $\m\otimes\mathrm{H}^3$ and $\m\otimes\mathrm{W}^3$ respectively.
Set 
$$\mathrm{MC}'_m(\m)=\left\{\xi\in\m\otimes\mathrm{C}^2\left(\g,\g\right):d(\xi)-\frac{1}{2}\left[\xi,\xi\right]_{\mathrm{B}^3}=0\right\},$$
we can see $\mathrm{MC}_m(\m)\subset\mathrm{MC}'_m(\m)$. 
From (\ref{e3.1}) we may choose a basis  $(e^k_\alpha)_{\alpha\in\mathcal{I}^k}$ of $\mathrm{C}^k(\g,\g)$ such that the elements $e^k_\alpha$ are indexed by $\mathcal{H}^k$, $\mathcal{B}^k$ and $\mathcal{W}^k$, and form a basis of $\mathrm{H}^k$, $\mathrm{B}^k(\g,\g)$ and $\mathrm{W}^k$ respectively. Denote by $|\mathcal{A}|$ the cardinal of a set $\mathcal{A}$, we have $|\mathcal{W}^k|=|\mathcal{B}^{k+1}|$. Set $p:=|\mathcal{W}^2|=|\mathcal{B}^3|.$
\begin{lem}\label{l2.1}There is a unique map  $g:\m\otimes\mathrm{Z}^2(\g,\g)\rightarrow\m\otimes\mathrm{W}^2(\g,\g)$ with $g(0)=0$ such that $\mathrm{MC}_m'(\m)=\left\{\xi=z+g(z):z\in\m\otimes\mathrm{Z}^2(\g,\g)\right\}$.
\end{lem}
\begin{proof} Consider a map $F:\m\otimes\mathrm{Z}^2(\g,\g)\times\m\otimes\mathrm{W}^2\rightarrow\m\otimes\mathrm{B}^3(\g,\g)$ defined by $F(z,w)=d(w)-\frac{1}{2}\left[z+w,z+w\right]_{\mathrm{B}^3}$.
Denote by  $(F^\alpha)_{\alpha\in\mathcal{B}^3}$, $(z^\alpha)_{\alpha\in\mathcal{Z}^2}$ and $(w^\alpha)_{\alpha\in\mathcal{W}^2}$ the components of $F$, $z$ and $w$ relative to above bases. Since $F$ is polynomial in variables $z^\alpha$ and $w^\alpha$, we may regard  $F$ as an element of the ring $\A[[Z,W]]^{p}$ of formal power series in variables $Z=(Z^\alpha)_{\alpha\in\mathcal{Z}^2}$ and $W=(W^\alpha)_{\alpha\in\mathcal{W}^2}$. We have $F(0)=0$. The Jacobian $\left[\partial F^\alpha/\partial{W^\beta}(0)\right]$ is invertible in $\A$, since $\mathrm{pr}(\left[\partial F^\alpha/\partial {W^\beta}(0)\right])$ is invertible in $\A/\m=\K$ (i.e. $d_{\phi_0}:\mathrm{W}^2\rightarrow\mathrm{B}^3$ is an isomorphism).
It follows from above and  the formal implicit theorem, the formal equation $F(Z,W)=0$ admits a solution if and only if there is a unique formal map $G=(G^\alpha)_{\alpha\in\mathcal{W}^2}\in\A[[Z]]^{p}$ such that $G(0)=0$ and $G(Z)=W$. It is clear that $G$ converges in the $\m$-adic sense. Since $\A$ is complete, it follows that we define by substituting, a map $g^\alpha:\A^{|\mathcal{Z}^2|}\rightarrow\A^{p}$ by $g^\alpha(a_1,...,a_{|\mathcal{Z}^2|}):=G^\alpha(a_1,...,a_{|\mathcal{Z}^2|})$ which satisfies $g^\alpha(\m^{|\mathcal{Z}^2|})\subset\m^{p}$, for all $\alpha$.
\end{proof}\\
It follows from Lemma \ref{l2.1} that we can define a map
$$\Omega:\m\otimes\mathrm{Z}^2(\g,\g)\rightarrow\m\otimes\mathrm{H}^3,\quad z\mapsto\left[z+g(z),z+g(z)\right]_{\mathrm{H}^3},$$
which is called obstruction map.
\begin{thm}\label{c3.1} There are $g$ and $\Omega$ defined as above
satisfying $g(0)=\Omega(0)=0$ such that 
$\Def(\phi_0,\A)=\left\{\phi_0+z+g(z):\Omega(z)=0,z\in\m\otimes\mathrm{Z}^2(\g,\g)\right\}.$
\end{thm}
\begin{proof} The element $\xi$ belongs to  $\mathrm{MC}_m(\m)$ iff $d(\xi)-\frac{1}{2}\left[\xi,\xi\right]_{\mathrm{B}^3}=0,\left[\xi,\xi\right]_{\mathrm{W}^3}=0,\left[\xi,\xi\right]_{\mathrm{H}^3}=0.$ 
We shall show that if $\xi$ satisfies $\left[\xi,\xi\right]_{\mathrm{H}^3}=0=d(\xi)-\frac{1}{2}\left[\xi,\xi\right]_{\mathrm{B}}$ then $\left[\xi,\xi\right]_{\mathrm{W}^3}=0$. We suppose that $\left[\xi,\xi\right]_{\mathrm{H}^3}=0=d(\xi)-\frac{1}{2}\left[\xi,\xi\right]_{\mathrm{B}^3},$
since 
$0=\left[\phi,\left[\phi,\phi\right]\right]=\left[\phi,d(\xi)-\frac{1}{2}\left[\xi,\xi\right]\right]$ 
and by Eq. (\ref{e3.3}), it follows that $\left[\phi,\left[\xi,\xi\right]_{\mathrm{W}^3}\right]=0$. We can check that $\left[\phi,\left[\xi,\xi\right]_{\mathrm{W}^3}\right]=(d+ad\xi)(\left[\xi,\xi\right]_{\mathrm{W}^3}),$ 
(the graded Jacobi's identity) and the injectivity of $d:\mathrm{W}^2\rightarrow\mathrm{C}^3(\g,\g)$ implies $ad \phi:\mathrm{W}^2\rightarrow\m\otimes\mathrm{C}^3(\g,\g)$ does. It follows that $\left[\xi,\xi\right]_{\mathrm{W}^3}=0$.
By Lemma \ref{l2.1}, $\xi$ is equal to $z+g(z)$ is equivalent to $d(\xi)-\frac{1}{2}\left[\xi,\xi\right]_{\mathrm{B}^3}=0$ and we deduce that $\mathrm{MC}_m(\m)=\left\{z+g(z):\Omega(z)=0,z\in\m\otimes\mathrm{Z}^2(\g,\g)\right\}$ and $\Def(\phi_0,\A)=\phi_0+\mathrm{MC}_m(\m)$.
\end{proof}
\subsection{Versal Deformations}
The Zariski's tangent space of the local $\K$-algebra $\A$  with maximal ideal $\m$ is the dual $(\m/\m^2)^*$ of the vector space $\m/\m^2$. 
\begin{de}\label{V}A deformation $f:\O\rightarrow\mathrm{R}$ with $\mathrm{R}\in\widehat{\got{R}}$, is (formal) versal if \\
\begin{enumerate}
	\item [(i)] for
any deformation $h:\O\rightarrow\mathrm{A}$ with $\A\in\got{R}$, there is a local morphism $\overline{h}:\mathrm{R}\rightarrow\mathrm{A}$ (a base change) such that $h\circ f$ is equivalent to $h$,
	\item[(ii)]the Zariski's tangent space of $\mathrm{R}$ is minimal for the property $(i)$ i.e. 
	$$\dim_\K(\m(\mathrm{R})/\m^2(\mathrm{R}))=\dim_\K(\mathrm{H}^2(\g,\g))$$
\end{enumerate}
	If the maximal ideal $\m$ of $\A$ in $(i)$ satisfies $\m^2=0$, $\pi$ is called miniversal in \cite{FF}. 
\end{de} 
 
If $\theta^\alpha$ is the $\alpha$-th component of the map $\theta:\Gl\left(\V\right)\rightarrow\mathrm{C}^2\left(\V,\V\right)$ defined by $\theta\left(s\right)=s\ast\phi_0$, on the basis $(e_{\alpha})_{\alpha\in\mathcal{I}}$ of $\mathrm{C}^2\left(\V,\V\right)$, then its differential at unit $1$ is given by
$
	(D_1\theta^\alpha)\left(L\right)=\left(L\cdot\phi_0\right)^\alpha=-(dL)^\alpha,L\in\mathrm{C}^1\left(\V,\V\right),
$
since  $\theta^\alpha\left(1+tL\right)=\left(\left(1+tL\right)\ast\phi_0\right)^\alpha=\phi^\alpha_0+t\left(L\circ\phi_0-\phi_0\left(L,-\right)-\phi_0(-,L\right))^\alpha\mod t^2$.
\begin{de} A set $\mathcal{A}$ is said to be admissible at point $\phi_0$ if it is a minimal subset of $\mathcal{I}$ such that the rank of $(D_1\theta^\alpha)_{\alpha\in\mathcal{A}}$ is maximal, i.e. $|\mathcal{A}|=\mathrm{rank}\left(d_{\phi_0}\right)=\dim\mathrm{B}^2\left(\g,\g\right)=\dim\left[\phi_0\right]$.
\end{de}
\begin{rem}\label{R2.1}Let $(e_\alpha)_{\alpha\in\mathcal{I}}$ be a basis of $\mathrm{C}^2\left(\g,\g\right)$. From the exchange basis theorem, there are parts \,$\mathcal{B}:=\mathcal{I}-\mathcal{A}\subset\mathcal{I}$ such that the $(e_\beta)_{\beta\in\mathcal{B}}$ complete a given basis of $\mathrm{B}^2(\g,\g)$ to a basis of $\mathrm{C}^2\left(\g,\g\right)$. The set $\mathcal{A}$ is the complement of $\mathcal{B}$ in $\mathcal{I}$.
\end{rem}
Let $(\phi^{\alpha})=((\phi_0)^\alpha)+(\xi^\alpha)$ be a deformation of $\phi_0$, with $\xi^\alpha\in\m(\A)$. The parameters $\left(\xi^\alpha\right)_{\alpha\in\mathcal{A}}$ are called orbital parameters at $\phi_0$. The set $\mathcal{A}$ permits to define a subscheme of $\L_m$ (\ref{e1.8}) as $\got{F}:=\L^\mathcal{A}_{m,\phi_0},$
where the components $\xi^{\alpha}$, $\alpha\in\mathcal{A}$, are expressed as $\K$-linear combinations of $X^\alpha-(\phi_0)^{\alpha}$.
The tangent space $T_{\phi_0}\left(\got{F}\right)$ is isomorphic to $\mathrm{H}^2(\g,\g)$ and is transversal at the orbit $\left[\phi_0\right]$ under $\G_m(\A)$ at $\phi_0\in\got{F}$ i.e. the tangent spaces at $\phi_0$ satisfy 
$T_{\phi_0}\L_m=T_{\phi_0}\left[\phi_0\right]\oplus T_{\phi_0}\left(\got{F}\right),$
where the orbit is a reduced scheme, and $\got{F}$ is called \textbf{slice} associated with $\mathcal{A}$.
The set of deformations of $\phi_0$ with base $\A$ in the slice $\L^\mathcal{A}_{m,\phi_0}$ is given by
$$
	\Def_\mathcal{A}(\phi_0,\A)=\left\{\phi\in\L^\mathcal{A}_{m,\phi_0}(\A):\mathrm{pr}(\phi)=\phi_0\right\}.
$$
The functor 
$\Def_\mathcal{A}(\phi_0-)$ is representable by the local ring $\O^\mathcal{A}_{\phi_0}$  of $\L^{\mathcal{A}}_{m,\phi_0}$ at $\phi_0$.
If $\mathcal{A}$ is an admissible set of $\mathcal{I}$ at $\phi_0$, then the vector subspace $\mathrm{V}^2_{\mathcal{A}}$ of $\mathrm{C}^2(\g,\g)$ generated by $(e_\alpha)_{\alpha\in\mathcal{I}-\mathcal{A}}$ 
is a linear complement  of $\mathrm{B}^2(\g,\g)$ in $\mathrm{C}^2(\g,\g)$ containing $\mathrm{H}^2$, see Remark \ref{R2.1}.
We have
\begin{equation}
	\Def_\mathcal{A}(\phi_0,\A)=\left\{\phi_0+v;d v-\frac{1}{2}\left[v,v\right]=0,v\in\m\otimes\V^2_{\mathcal{A}}\right\}.
\end{equation}
\begin{prop}\label{p2.2}For any $\A\in\widehat{\got{R}}$ and an admissible set $\mathcal{A}$ at $\phi_0$, then there are two maps $g$ and $\Omega$ defined as above
satisfying $g(0)=\Omega(0)=0$ such that 
$$\Def_{\mathcal{A}}(\phi_0,\A)=\left\{\phi_0+h+g(h);\Omega(h)=0,h\in\m\otimes\mathrm{H}^2\right\}.$$
\end{prop}
\begin{proof} See Theorem \ref{c3.1}
\end{proof}
\begin{lem}\label{t3.2}If $\mathcal{A}$ is an admissible set at $\phi_0$, then the map
$$f:(w,v)\mapsto (\mathrm{id}+w)\ast(\phi_0+v)-\phi_0$$ 
is a bijection from $ (\m\otimes\mathrm{W}^1)\times(\m\otimes\mathrm{V}^2_{\mathcal{A}}(\g,\g))$ to $\m\otimes\mathrm{C}^2(\g,\g)$.
\end{lem}
\begin{proof} 
Let $(f^\alpha)_{\alpha\in\mathcal{I}}$, $(v^\alpha)_{\alpha\in\mathcal{I}-\mathcal{A}}$ and $(w^\alpha)_{\alpha\in\mathcal{W}^1}$ be the components of $f$, $v$ and $w$ relative to above bases (see Section 2.2). We may extend $f$ to an element $F$ of $\A[[V,W]]^{|\mathcal{I}|}$ the ring of formal power series in indeterminates  $V=(V^\alpha)_{\alpha\in\mathcal{I}-\mathcal{A}}$ and $W=(W^\alpha)_{\alpha\in\mathcal{W}^1}$. We have $F(0)=0$. The Jacobian $\mathrm{J}:=\left[(F^\alpha/\partial {V^{\beta}},F^{\alpha'}/\partial{W^{\beta'}})(0)\right]$ is given by	
$\mathrm{J}$=\begin{math}\bordermatrix{&\cr        
  & \mathrm{I}_{\mathcal{I}-\mathcal{A}}& 0 \cr          
  & 0  & d_{{\phi_0}}  \cr  }
  \end{math}
, where $\mathrm{I}_{\mathcal{I}-\mathcal{A}}$ is the identity matrix of size $|\mathcal{I}-\mathcal{A}|$.
It is easy to check that $\mathrm{J}$ is invertible in $\A$, since its projection $\mathrm{pr}(\mathrm{J})$ is inversible in $\A/\m=\K$ ($d_{\phi_0}$ is an isomorphism from $\mathrm{W}^1$ to $\mathrm{B}^2(\g,\g)$). It follows from above and the formal inversion theorem \cite{B} that the formal map $F$ is bijective. It is clear that $F$ converges in the $\m$-adic sense. Since $\A$ is complete, it follows that $f$ coincides with the associated map $F$ by substituting the indeterminates $V_\alpha$ and $W_\alpha$ for the elements of $\A$.
\end{proof}
\begin{thm}If $\mathcal{A}$ is an admissible set at $\phi_0$, then the map
$$F:(w,\phi)\mapsto (\mathrm{id}+w)\ast\phi$$ 
is a bijection from $(\m\otimes\mathrm{W}^1)\times\Def_{\mathcal{A}}(\phi_0,\A)$ to $\Def(\phi_0,\A)$.
\end{thm}
\begin{proof} Let $\phi'=\phi_0+\eta'\in\Def(\phi_0,\A)$ with $\eta'\in\m\otimes\mathrm{C}^2(\g,\g)$. By Lemma \ref{t3.2} there are unique elements $w\in\m\otimes\mathrm{W}^1$ and $v\in\m\otimes\mathrm{V}^2_{\mathcal{A}}(\g,\g)$ such that $\eta'=s\ast(\phi_0+v)-\phi_0$ i.e. $\phi'=s\ast(\phi_0+v)$, with $s=\mathrm{id}+w$. Hence $\phi:=\phi_0+v=s^{-1}\ast\phi'$ and $\phi\in\Def(\phi_0,\A)\cap(\phi_0+\m\otimes\mathrm{V}^2_{\mathcal{A}}(\g,\g))=\Def_{\mathcal{A}}(\phi_0,\A).$
\end{proof}
\begin{cor}\label{c3.2}Let $\mathrm{A}$ be an element $\widehat{\got{R}}$ and $\mathcal{A}$ an admissible set at $\phi_0$. 
For all deformation $\phi'\in\Def(\phi_0,\A)$, there exists a $\phi$ in the orbit $[\phi']$ under $\G_m(\A)$ such that $\phi^\alpha=(\phi_0)^\alpha$ for all $\alpha\in\mathcal{A}$. 
\end{cor}
The local ring $\O_{\phi_0}^\mathcal{A}$ of $\L^\mathcal{A}_{m,\phi_0}$ at $\phi_0$ is the quotient of the local ring $\O_{\phi_0}$ of $\L_m$ at $\phi_0$ by the ideal $\sum_{\alpha\in\mathcal{A}}\O\cdot\xi^\alpha$ generated by the $\xi^\alpha\in\m\left(\O\right)$ where $\alpha\in\mathcal{A}$. The canonical deformation $\mathrm{id}:\O^\mathcal{A}_{\phi_0}\rightarrow\O^\mathcal{A}_{\phi_0}$ of $\phi_0$ in $\L^{\mathcal{A}}_{m,\phi_0}$ defined  by  $y=(y^\alpha)_{\alpha\in\mathcal{I}-\mathcal{A}}$, is the projection of the canonical deformation of $\phi_0$ in $\L_m$ defined by $x=(x^\alpha)_{\alpha\in\mathcal{I}}$,
\begin{equation}\label{e2.8}
	\xymatrix{
	\O_{\phi_0} 
	\ar@{->}[r]^{\mathrm{id}}
	\ar@{->}[d]_{\pi}
	& \O_{\phi_0}
		\ar@{->}[d]^{\pi}
	\\
	\O^\mathcal{A}_{\phi_0}\ar@{->}[r]_{\mathrm{id}}&	\O^\mathcal{A}_{\phi_0}}
	\end{equation}
We use the notation of Proposition \ref{p2.2}.	
\begin{cor}\label{N21}For any  admissible set $\mathcal{A}$ at $\phi_0$, then 
\begin{enumerate}
	\item $\pi:\O\rightarrow\O^\mathcal{A}_{\phi_0}$ is versal.
\item The canonical deformation $\mathrm{id}$ of $\phi_0$ in $\L^{\mathcal{A}}_{m,\phi_0}$ defined by (\ref{e2.8}) may be written on the completing local ring $\widehat{\O^\mathcal{A}_{\phi_0}}$ as
$y=\phi_0+(y_\alpha)+g(y_\alpha)$ such that $\Omega((y_\alpha))=0,$
where $\alpha$ runs through the set $\mathcal{H}^2$.
The parameters $(y_\alpha)_{\alpha\in\mathcal{H}^2}$ are said to be essential.
\end{enumerate}
\end{cor}
\begin{proof} 1. By Corollary \ref{c3.2}, each deformation $h:\O_{\phi_0}\rightarrow\A$ with $\A\in\widehat{\got{R}}$ is equivalent to a deformation $h_0:\O_{\phi_0}\rightarrow\A$ such that $h_0\left(x^\alpha\right)=(\phi_0)^\alpha$ for all $\alpha\in\mathcal{A}$. Then $h_0$ vanishes on the ideal generated by the elements $x^\alpha-(\phi_0)^\alpha$, $\alpha\in\mathcal{A}$, so it factorizes through $\O^\mathcal{A}_{\phi_0}$. There is a local morphism $\overline{h}_0:\O^\mathcal{A}_{\phi_0}\rightarrow\A$ such that $h_0=\overline{h}_0\circ\pi$.
It is clear that $\dim_\K(\m(\O^{\mathcal{A}}_{\phi_0})/\m^2(\O^{\mathcal{A}}_{\phi_0}))=\dim_\K(\mathrm{H}^2(\g,\g))$. The statement 2 is coming from Proposition \ref{p2.2}.
\end{proof}
\begin{rem}\label{r2.1} The results in this section remain true in $\L_n^{\mathrm{R}}$ by replacing $\G_m(\A)$, $\phi_0$, $\O_{\phi_0}$, $\O^{\mathcal{A}}_{\phi_0}$ and $\mathrm{C}^k(\g,\g)$ to $\G_n(\A)^{\mathrm{R}}$, $\varphi_0$, $\O^\mathrm{R}_{\varphi_0}$, $\O^{\mathrm{R},\mathcal{A'}}_{\varphi_0}$, $\mathrm{C}^k(\n,\n)^{\mathrm{R}}$ respectively. Denote by $\Omega'$, $g'$, $F'$, $\mathcal{I'}$ and $\mathcal{A'}$ the corresponding elements.  
\end{rem}
\begin{center}
Algorithm of construction of a versal deformation of $\phi_0$.
\end{center}
\begin{enumerate}
	\item Determine an admissible set $\mathcal{A}$ at $\phi_0$, a minimal subset of $\mathcal{I}$ such that the rank of $(D_1\theta^\alpha)_{\alpha\in\mathcal{A}}$ is maximal, i.e. $|\mathcal{A}|=\mathrm{rank}\left(d_{\phi_0}\right)$.
	\item Fix $X^\alpha$ to $(\phi_0)^\alpha$ for each $\alpha\in\mathcal{A}$ in the Jacobi's equations  $(\ref{n1})$ in coordinates $X=(X^\alpha)_{\alpha\in\mathcal{I}}$ defining the scheme $\L_m$. This corresponds to quotient by the ideal spanned by $X^\alpha-(\phi_0)^\alpha, \alpha\in\mathcal{A}$. We obtain the Jacobi's equations in the quotient coordinates $\overline{X}=(\overline{X}^\alpha)_{\alpha\in\mathcal{B}=\mathcal{I}-\mathcal{A}}$, denoted by $(2).$
	\item Consider the linear equations $(3)$ with coefficients in $\K$  of $(2)$. Solve Eq. $(3)$ in function of an arbitrary choice of variables $(\overline{X}^\lambda)$ indexed by a set $\mathcal{H}$, of cardinal $\dim_\K\mathrm{H}^2(\g,\g)$. This choice permits us to build a power series $g$ in coordinates $(\overline{X}^\lambda)_{\lambda\in\mathcal{H}}$ such that $\overline{X}=(\overline{X}^\lambda)_{\lambda\in\mathcal{H}}+g((\overline{X}^\lambda)_{\lambda\in\mathcal{H}})$.
	\item The versal deformation of $\phi_0$ associated with $\mathcal{A}$ in $\L_m$ is given by: \\
	$(\phi_0)^\alpha$, for $\alpha\in\mathcal{A}$, and $\overline{X}^\alpha$, for $\alpha\in\mathcal{B}$.
	\item Eq $(2)$ defines the scheme $\L_m^\mathcal{A}$, and the local ring $\O^{\mathcal{A}}_{\phi_0}$ is isomorphic to the ring of rational functions $\frac{P(\overline{X})}{Q(\overline{X})}$, where $
	P,Q\in\K[\overline{X}]$, and $Q(\phi_0)\neq 0$.
\end{enumerate}
\section{Reduction Theorem and Formal Rigidity}
\subsection{Reduction Theorem}
Let $\g=(\mathrm{R}\ltimes\n,\phi_0)$ be an algebraic Lie algebra with $(\n,\varphi_0)$ the nilpotent radical of dimension $n$ and $\mathrm{R}=\mathrm{U}\oplus\mathrm{S}$ a maximal reductive Lie subalgebra. To construct a versal deformation of a point in $\L_m$ is a hard problem, but in the algebraic case this problem can be reduced from that of $\varphi_0$ in $\L_n^{\mathrm{R}}$, under some conditions. The equality of the number of essential parameters of $\phi_0$ in $\L_m$ and $\varphi_0$ in $\L_n^{\mathrm{R}}$ is a necessary condition according to Corollary \ref{N21} and Remark \ref{r2.1}.\\
The graded Lie algebra morphism $\mathrm{i}:\mathrm{C}\left(\n,\n\right)^\mathrm{R}\rightarrow\mathrm{C}\left(\g,\g\right)$ is defined by the composition of graded differential complex morphisms $\alpha$ and $\beta$ defined by
\begin{equation}\label{C}
	0\rightarrow\mathrm{C}\left(\n,\n\right)^\mathrm{R}\stackrel{\alpha}{\rightarrow}\mathrm{C}\left(\n,\g\right)^\mathrm{R}\stackrel{\pi}{\rightarrow}\mathrm{C}\left(\n,\g/\n\right)^\mathrm{R}\rightarrow 0,
\end{equation}
and $\beta:\mathrm{C}\left(\n,\g\right)^\mathrm{R}\stackrel{}{\rightarrow}\mathrm{C}\left(\g,\g\right)$
defined  by $\mathrm{i}\left(f\right)|_{\mathrm{R}\times\g^{q-1}}=0\quad\mathrm{and}\quad\mathrm{i}\left(f\right)|_{\n^{q}}=f,$
for each $f\in\mathrm{C}^q\left(\n,\n\right)^\mathrm{R}$.
This lemma is left to the reader.
\begin{lem}\label{n31} For all $f$ and $g$ of $\mathrm{C}\left(\n,\n\right)^\mathrm{R}$, we then have\\
	$i\left(\left[f,g\right]\right)=\left[i\left(f\right),i\left(g\right)\right]$ and $i\left(\left[\varphi_0,f\right]\right)=\left[\phi_0,i\left(f\right)\right]$.
\end{lem}
We deduce the linear morphism of cohomologies $\overline{\mathrm{i}}=\oplus_q\overline{\mathrm{i}_q}:\mathrm{H}\left(\n,\n\right)^\mathrm{R}\rightarrow\mathrm{H}\left(\g,\g\right).$
The groups $\Gl\left(\n\right)_0^\mathrm{R}$ and $\Gl\left(\g\right)$ canonically act on $\mathrm{C}\left(\n,\n\right)^\mathrm{R}$ and $\mathrm{C}\left(\g,\g\right)$  by $*$. We define an injection, again denoted by $\mathfrak{I}$, from the group $\Gl\left(\n\right)_0^\mathrm{R}$ into $\Gl\left(\g\right)$ which sends $s\in\Gl\left(\n\right)_0^\mathrm{R}$ into $\mathfrak{I}(s)$ with $\mathfrak{I}(s)|_\n=s$ and $\mathfrak{I}(s)|_\mathrm{R}=\mathrm{id}_{\mathrm{R}}$.
\begin{lem}\label{n32} Let $\varphi_0\in\L_n^\mathrm{R}\left(\K\right)$ and $\phi_0=\mathfrak{I}(\varphi_0)$. Then
$i\left(s\ast f\right)=\mathfrak{I}(s)\ast i\left(f\right)$, for all $s\in\Gl\left(\n\right)_0^\mathrm{R}$, $f\in\mathrm{C}\left(\n,\n\right)^\mathrm{R}$, and the following diagram is commutative
	\begin{equation}
	\xymatrix{
	\varphi_0+f 
	\ar@{->}[r]^{\mathfrak{I}}
	\ar@{->}[d]_{s\ast}
	& \phi_0+\mathrm{i}(f)
		\ar@{->}[d]^{\mathfrak{I}(s)\ast}
	\\
	s\ast(\varphi_0+f)\ar@{->}[r]_{\mathfrak{I}}&	\mathfrak{I}(s)\ast(\phi_0+\mathrm{i}(f))}
	\end{equation}
\end{lem}
\begin{proof} The equality $\mathfrak{I}(s)\ast\phi_0=\mathfrak{I}(s\ast\varphi_0)$ comes from the commutativity of $\mathfrak{I}(s)$ with $\mathrm{R}$. The second statement is deduced from the first.
\end{proof}\\
The map $\mathrm{i}$ extends  to a graded Lie algebra morphism $\mathrm{id}\otimes\mathrm{i}$ from $\A\otimes\mathrm{C}\left(\n,\n\right)^\mathrm{R}$ to $\A\otimes\mathrm{C}\left(\g,\g\right)$, denoted again by $\mathrm{i}$, satisfying similar lemmas \ref{n31} and \ref{n32}. From Lemma \ref{n31}, if $\varphi=\varphi_0+\xi\in\Def_n^\mathrm{R}(\varphi_0,\A)$ 
then $\mathrm{i}(d\xi-\frac{1}{2}\left[\xi,\xi\right])=d\mathrm{i}(\xi)-\frac{1}{2}\left[\mathrm{i}(\xi),\mathrm{i}(\xi)\right]=0,$ 
i.e. $\phi_0+\mathrm{i}(\xi)\in\Def_m(\phi_0,\A)$. From Lemma \ref{n32}, the scheme morphism $\mathfrak{I}:\L_n^\mathrm{R}\rightarrow\L_m$ induces a map $\varphi_0+\xi\mapsto\phi_0+\mathrm{i}(\xi)$ from $\Def_n^\mathrm{R}(\varphi_0,\A)$ to $\Def_m(\phi_0,\A)$ denoted again by $\mathfrak{I}$. It passes to the quotient $$\overline{\mathfrak{I}}:\overline{\Def}_n^\mathrm{R}\left(\varphi_0,\A\right)\longrightarrow\overline{\Def}_m\left(\phi_0,\A\right),$$
modulo the actions of the groups
$\G_n\left(\A\right)^\mathrm{R}$ and $\G_m\left(\A\right)$ 
respectively.
\begin{lem}\label{n33}We then have
\begin{enumerate} 
	\item $\mathrm{Z}\left(\g,\g\right)\cap 	\mathrm{i}(\mathrm{C}\left(\n,\n\right)^\mathrm{R})=\mathrm{i}(\mathrm{Z}\left(\n,\n\right)^\mathrm{R})$.
	\item $\mathrm{B}\left(\g,\g\right)\cap \mathrm{i}(\mathrm{C}^{p+1}\left(\n,\n\right)^\mathrm{R})=\mathrm{i}(\mathrm{B}^{p+1}(\n,\n)^{\mathrm{R}}+\delta\mathrm{Z}^p(\n,\g/\n)^{\mathrm{R}})$, 
\end{enumerate}
where $\delta$ defines the connection $\partial$ in the long sequence of the cohomology .
\end{lem}
\begin{proof} 1. The statement $d f=0$ for all $f\in\mathrm{C}\left(\n,\n\right)^\mathrm{R}$ is equivalent to $\mathrm{i}\left(d f\right)=d i\left(f\right)=0$, i.e. $\mathrm{i}\left(f\right)\in\mathrm{Z}\left(\g,\g\right)$ since $i$ is injective.\\
2. Given $d h\in\mathrm{B}^{p+1}\left(\g,\g\right)$ of the form $\mathrm{i}\left(f\right)$ with $f\in\mathrm{C}^{p+1}\left(\n,\n\right)^\mathrm{R}$, the $\mathrm{R}$-invariance implies
 $\mathrm{ad}_\g\mathrm{R}\cdot d h=d(\mathrm{ad}_\g\mathrm{R}\cdot h)=0,$ 
thus $\mathrm{ad}_\g\mathrm{R}\cdot h\subset\mathrm{Z}^{p}\left(\g,\g\right)$. It follows that $h$ may be written as $h_0+h_1$ with $h_0\in\mathrm{C}^{p}\left(\g,\g\right)^\mathrm{R}$ and $h_1\in\mathrm{Z}^{p}\left(\g,\g\right)$, and $d h_0=\mathrm{i}\left(f\right)$. Denote by $\rho\left(x\right)$ the operator defined by 
$
	\left(\rho\left(x\right)h_1\right)\left(x_1,...,x_p\right):=h_1\left(x,x_1...,x_p\right),
$
where $x,x_1,...,x_p\in\g$ and $\theta$ is the adjoint representation of $\g$ in $\mathrm{C}\left(\g,\g\right)$. It is well-known that $\theta$ and $\rho$ satisfy the formula 
$
	d\circ\rho\left(x\right)+\rho\left(x\right)\circ d=\theta\left(x\right),\quad\forall x\in\g.
$
Then $d\rho\left(x\right)h_0=0, \forall x\in\mathrm{R}$. We define a map $\widetilde{h}_0\in\mathrm{C}^{p}\left(\g,\g\right)^\mathrm{R}$ by 
$\widetilde{h}_0|_{\n^p}=0\quad\mathrm{and}\quad\rho\left(x\right)\widetilde{h}_0=\rho\left(x\right)h_0,$ 
for all $x\in\mathrm{R}$. Hence $\widetilde{h}_0$ is a cocycle since we have $d\widetilde{h}_0=0$ on $\n^{p+1}$ by construction and , for all $x\in\mathrm{R}$, we have 
$\rho\left(x\right)(d\widetilde{h}_0)=-d\rho\left(x\right)\widetilde{h}_0+\theta\left(x\right)\widetilde{h}_0=-d(\rho\left(x\right) \widetilde{h}_0)=0.$
It is clear that $\rho(x)(h_0-\widetilde{h}_0)(x)=0$ for all $x\in\mathrm{R}$, then
\begin{equation}\label{e3.4}
	h_0-\widetilde{h}_0=\mathrm{i}(k)+l,
\end{equation}
with $k\in\mathrm{C}^p(\n,\n)^{\mathrm{R}}$, $l=\mathrm{pr}\circ(	h_0-\widetilde{h}_0)$ and $\mathrm{pr}:\g\rightarrow\mathrm{R}$ the projection. It follows 
$
	d h=d(h_0-\widetilde{h}_0)=\mathrm{i}(d k)+d l,
$
since $\mathrm{i}\circ d=d\circ\mathrm{i}$. We deduce that $d h$ belongs to the image of $\mathrm{i}$ if and only if is $d l$, i.e. 
\begin{equation}\label{e3.6}
	\mathrm{pr}\circ d l=0.
\end{equation}
Since $l$ is $\mathrm{R}$-invariant (see (\ref{e3.4})), it is easy to prove that $\rho(x) l=0$ for all $x\in\mathrm{R}$. Hence
$
	\rho(x) d l=-d \rho(x) l+\theta(x)d l=0,
$
for all $x\in\mathrm{R}$. For all $(x_1,...,x_{p+1})\in\n^{p+1}$, 
\begin{equation}\label{e3.8}
	(\mathrm{pr}\circ d l)(x_1,...,x_{p+1})=\sum_{i<j}(-1)^{i+j}l([x_i,x_j],...,\widehat{x}_i,...,\widehat{x}_j,..x_{p+1}),
\end{equation}
which is the differential of the restriction $l|_{\n^p}$ on $\n^p$ in the complex $\mathrm{C}(\n,\mathrm{R})^{\mathrm{R}}$ with value in the trivial $\n$-module $\mathrm{R}$ which is identified with the adjoint action $\n$ in $\g/\n$. It follows from (\ref{e3.6}) and (\ref{e3.8}) that 
$
	l|_{\n^p}\in\mathrm{Z}^p(\n,\g/\n)^{\mathrm{R}}
$
such that $(	d l)|_{\n^{p+1}}=\delta (l|_{\n^p})\in\mathrm{Z}^{p+1}(\n,\n)^{\mathrm{R}}$, where $\delta$ defines the connection $\partial$ in the long sequence of cohomology. Then
$
	d h=\mathrm{i}(d k+\delta(l|_{\n^p}))\in\mathrm{i}(\mathrm{B}^{p+1}(\n,\n)^{\mathrm{R}}+\delta\mathrm{Z}^p(\n,\g/\n)^{\mathrm{R}}).
$
\end{proof}
\begin{cor}\label{C3.1}
\begin{enumerate}
	\item $\overline{\mathrm{i}}_k$ is injective iff $\mathrm{B}^k(\g,\g)\cap\mathrm{i}_k(\mathrm{C}^k(\n,\n)^{\mathrm{R}})=\mathrm{i}_k(\mathrm{B}^k(\n,\n)^{\mathrm{R}}).$
	\item  $\overline{\mathrm{i}}_k$ is an epimorphism iff 
	$\mathrm{Z}^k(\g,\g)=\mathrm{i}_k(\mathrm{Z}^k(\n,\n)^{\mathrm{R}})+\mathrm{B}^k(\g,\g).$
\end{enumerate}
\end{cor}
Let $\phi\in\L_{m}(\A)$, we define $\Aut\left(\phi,\A\right)$ the group of automorphisms of $\phi$ as the set of matrices $s\in \Gl_{m}\left(\A\right)$ such that $s\ast\phi=\phi$, and $\Der(\phi,\A)$ the Lie $\A$-algebra of derivations of $\phi$ as the set of matrices $\delta\in\mathrm{M}_m(\A)$ such that $\delta.\phi=0$.
\begin{lem}\label{n6}Let $\phi\in\Def\left(\phi_0,\A\right)$ be a deformation. If $\delta\in\Der(\phi,\A)$ then $\delta$ is a deformation of $\delta_0:=\mathrm{pr}(\delta)\in\Der(\phi_0,\K)$. If $\delta_0$ is an inner derivation of $\phi_0$ then it may be lifted to a derivation of $\phi$. 
\end{lem}
\begin{proof} We check that $\mathrm{pr}\left(\delta\right)\cdot\phi_0=0$ and $\delta_0=\mathrm{pr}\left(\delta\right)\in\Der\left(\phi_0,\K\right)$. If $x\in\g$ such that $\delta_0=\phi_0\left(x,.\right)$ then the map $\delta:=\phi\left(x,.\right)$ is a derivation of $\phi$.
\end{proof}
\begin{lem}\label{n7}Let $\phi\in\Def\left(\phi_0,\A\right)$ be a deformation. If $\delta$ is an element of $\m\otimes\Der(\phi,\A)$, then $\exp\left(\delta\right)=\sum^{\infty}_{n=0}\frac{1}{n!}\delta^n$ is an automorphism of $\phi$.
\end{lem}
\begin{proof}
Since $\A$ is complete, hence this formula converges in the Krull's sense  in $\M_{m}\left(\A\right)$. We can reason by induction on the integer $n$ and we will obtain	$\delta^n\circ\phi=\sum^{n}_{i=0}(^n_i)\phi(\delta^i(-),\delta^{n-i}(-)).$
\end{proof}
\begin{lem}\label{L}Let $s\in\G_m(\A)$ and $f:\O\rightarrow\A$ a deformation. If $f$ is surjective then $s\ast f$ is surjective.
\end{lem}
\begin{proof} The map $s\ast f:\O\rightarrow\A$ is defined by 
$$s\ast f(x^k_{ij})=(s\ast\phi)^k_{ij}=\sum_{lpq} (s^{-1})_{li}\,(s^{-1})_{pj}\,s_{kq}\,\phi^q_{lp}.$$ 
with $f(x^k_{ij})=\phi_{ij}^k$, see (\ref{B}). Then
\begin{equation}\label{CC}
	\phi^k_{ij}=\sum_{lpq} s_{li}\,s_{pj}\,(s^{-1})_{kq}\,(s\ast\phi)^q_{lp}=s\ast f(\sum_{lpq} \sigma_{li}\,\sigma_{pj}\,(\sigma^{-1})_{kq}\,x^q_{lp}).
\end{equation}
avec $s_{ij}=f(\sigma_{ij})$. The map $f$ being surjective and $(x^\alpha)_\alpha$ generates $\O$, then $(\phi^\alpha)_\alpha$ generates $\A$. We deduce from (\ref{CC}) that $s\ast f$ is surjective.
\end{proof}\\
\textbf{The Local Epimorphism $\overline{\eta}$.} 
We identify the vector space $\mathrm{C}^2(\n,\n)^{\mathrm{R}}$ with its image by $\mathrm{i}_2$. If $\overline{\mathrm{i}}_2$ is injective, by Corollary \ref{C3.1}, we then have the equality $\mathrm{B}^2(\g,\g)\cap\mathrm{i}_2(\mathrm{C}^2(\n,\n)^{\mathrm{R}})=\mathrm{i}_2(\mathrm{B}^2(\n,\n)^{\mathrm{R}}).$ 
Let $\mathrm{E}$ be the linear complement of $\mathrm{C}^2(\n,\n)^{\mathrm{R}}$ in $\mathrm{C}^2(\g,\g)$ generated by the $(e^k_{ij})$, with $i>n$, or $j>n$, or $k>n$, and the $\delta.e^k_{ij}$, with $\delta\in\mathrm{ad}\mathrm{R}|_\n$, and $i,j,k\leq n$, cf (\ref{E1.18}). We then have 
$\mathrm{B}^2(\g,\g)+\mathrm{E}=\mathrm{i}_2(\mathrm{B}^2(\n,\n)^{\mathrm{R}})+\mathrm{E}.$
If we choose a basis $(e_\alpha)_{\alpha\in\mathcal{I'}}$ of $\mathrm{C}^2(\n,\n)^{\mathrm{R}}$ which is completed to a basis $B$ of $\mathrm{C}^2(\g,\g)$ indexed by $\mathcal{I}$, containing all the $(e^k_{ij})$, with $i>n$, or $j>n$, or $k>n$ and some of $\delta.e^k_{ij}$, with $i$, $j$ and $k$ $\leq n$, then each admissible set $\mathcal{A}'\subset\mathcal{I}'$ at $\varphi_0$ will be completed to an admissible set $\mathcal{A}\subset\mathcal{I}$ at $\phi_0$, which will be contained in $\mathcal{A}'\cup(\mathcal{I}-\mathcal{I}')$. The scheme $\L_{n,\varphi_0}^{\mathrm{R},\mathcal{A}'}$ is the spectrum of the quotient of $\mathrm{P}_m$ by the ideal 
$$\J_{\mathcal{A}'}:=\J_m+\Delta_m+\left\langle X^k_{ij};i>n,\mathrm{or}\,j>n, \mathrm{or}\,k>n\right\rangle+\left\langle \xi^{\alpha}(X)-(\varphi_0)^\alpha;\alpha\in\mathcal{A}'\right\rangle,$$
where $\Delta_m$ is defined by (\ref{E1.14}), and the $\xi^{\alpha}(X)$ are $\K$-linear combinations of $X^k_{ij}$ associated with the basis change $(e^k_{ij})\rightarrow B$. The scheme $\L_{m,\phi_0}^{\mathcal{A}}$ is the spectrum of the quotient of $\mathrm{P}_m$ by the ideal
$\J_{\mathcal{A}}:=\J_m+\left\langle \xi^{\alpha}(X)-(\phi_0)^\alpha; \alpha\in\mathcal{A}\right\rangle.$
The identity map with the inclusion $\J_{\mathcal{A}}\hookrightarrow\J_{\mathcal{A}'}$, induces an epimorphism on the quotient algebras from $\mathrm{P}_m/\J_{\mathcal{A}}$ to $\mathrm{P}_m/\J_{\mathcal{A}'}$, a scheme embedding from $\L_{n,\varphi_0}^{\mathrm{R},\mathcal{A}'}$ to $\L_{m,\phi_0}^{\mathcal{A}}$ and finally, the local epimorphism  $\overline{\eta}:\O^{\mathcal{A}}_{\phi_0}\rightarrow\O_{\varphi_0}^{\mathrm{R},\mathcal{A}'}$.
\begin{thm}\label{N31}(\textbf{Reduction Theorem}). Let $\g:=\mathrm{R}\ltimes\n$ be an algebraic Lie algebra such that \\
$1.$ $\overline{\mathrm{i}}_1:\mathrm{H}^1\left(\n,\n\right)^\mathrm{R}\rightarrow\mathrm{H}^1\left(\g,\g\right)$ is an epimorphism,\\ 
$2.$ $\overline{\mathrm{i}}_2:\mathrm{H}^2\left(\n,\n\right)^\mathrm{R}\rightarrow\mathrm{H}^2\left(\g,\g\right)$  is an isomorphism, and \\
$3.$ $\overline{\mathrm{i}}_3:\mathrm{H}^3\left(\n,\n\right)^\mathrm{R}\rightarrow\mathrm{H}^3\left(\g,\g\right)$ is a monomorphism.
Then for every $\A\in\widehat{\got{R}}$, $\mathcal{A}\subset\mathcal{I}$ and $\mathcal{A'}\subset\mathcal{I'}$ admissible sets as above, the map $\mathfrak{I}$ from $\L_n^\mathrm{R}$ to $\L_m$ induces\\
(i) a bijection from $\overline{\Def}\left(\varphi_0,\A\right)^\mathrm{R}$ to $\overline{\Def}\left(\phi_0,\A\right)$, defined by $\left[\varphi\right]\longmapsto\left[\phi\right]$;\\
(ii) a local $\K$-algebra morphism, $\eta:\O_{\phi_0}\rightarrow\O^{\mathrm{R}}_{\varphi_0}$ which induces a local $\K$-algebra isomorphism, $\bar{\eta}:\O^{\mathcal{A}}_{\phi_0}\rightarrow\O^{\mathrm{R},\mathcal{A}'}_{\varphi_0}$ for $\K=\C$, such that the following diagram  is commutative
	\begin{equation}\label{e3.2}
	\xymatrix{
	\O_{\phi_0} 
	\ar@{->}[r]^{\eta}
	\ar@{->}[d]_{\pi}
	& \O^\mathrm{R}_{\varphi_0}
		\ar@{->}[d]^{\pi'}
	\\
\O^{\mathcal{A}}_{\phi_0}\ar@{->}[r]_{\bar{\eta}}&\O^{\mathrm{R},\mathcal{A}'}_{\varphi_0}}
	\end{equation}
\end{thm}
\begin{proof} 
\textbf{Surjection.} 
We shall show that for all deformation $\phi\in\Def\left(\phi_0,\A\right)$ there is $\varphi\in\Def\left(\varphi_0,\A\right)^\mathrm{R}$ such that $\phi$ and $\mathfrak{I}(\varphi)$ are equivalent under $\G_m(\A)$. 
We will prove for $\phi=\sum_{\alpha}t^{\alpha}\phi_{\alpha}$ by induction on the integer $p\in\N$ the following property:
there are $s_{p}=\mathrm{id}+\sum_{|\alpha|=p}t^{\alpha}s_{\alpha}\in\G_m\left(A\right)$ and $\Phi=\sum_{\alpha}t^{\alpha}\Phi_{\alpha}\in\Def\left(\phi_0,\A\right)$ such that
$\phi=s_p\ast\Phi,\quad \Phi_{\alpha}=\mathrm{i}_2(\varphi_{\alpha}),
$
where $\varphi_{\alpha}\in\mathrm{C}^2(\n,\n)^{\mathrm{R}}$ for $|\alpha|\leq p$.
It is obvious if $p=0$. We assume that $\Phi$ satisfies the induction hypothesis for $p$, the deformation equation can be expressed as
$
\left[\Phi,\Phi\right]=\sum_{\alpha_1}\sum_{\alpha_2}t^{\alpha_1}t^{\alpha_2}\left[\Phi_{\alpha_1},\Phi_{\alpha_2}\right]=0.
$
We get $	\left[\Phi,\Phi\right]=2\sum_{\alpha}t^\alpha (d\phi_\alpha-\omega_\alpha)=0,$
with $\omega_\alpha:=\frac{1}{2}\sum^{\alpha_1+\alpha_2=\alpha}_{\alpha_1\neq 0,\alpha_2\neq 0}[\phi_{\alpha_1},\phi_{\alpha_2}]$. 
It follows that,  
\begin{equation}\label{e39}
\sum_{|\alpha|\leq p+1} t^\alpha d\Phi_{\alpha}\equiv\sum_{|\alpha|\leq p+1} t^\alpha\omega_\alpha\mod t^{p+2}.
\end{equation}
From the equation (\ref{e39}), and the induction hypothesis, hence 
$$
	\sum_{|\alpha|=p+1} \overline{t}^{\alpha}d\Phi_{\alpha}=\sum_{|\alpha|= p+1}\overline{t}^{\alpha}\mathrm{i}_3\left(\frac{1}{2}\sum^{\alpha_1+\alpha_2=\alpha}_{\alpha_1\neq 0,\alpha_2\neq 0}\left[\varphi_{\alpha_1},\varphi_{\alpha_2}\right]\right)\in\m/\m^{p+2}\otimes\mathrm{i}_3(\mathrm{C}^3\left(\n,\n\right)^\mathrm{R}).
$$
We deduce from Lemma \ref{n33}(1) that $d_{\phi_0}\Phi_{\alpha}\in\mathrm{i}_3(\mathrm{Z}^3\left(\n,\n\right)^\mathrm{R})$  for all $|\alpha|=p+1$. 
Since $\overline{\mathrm{i}}_3$ is injective then  $d_{\phi_0}\phi_\alpha\in\mathrm{i}_3(\mathrm{B}^3(\n,\n)^{\mathrm{R}})$ by Corollary \ref{C3.1}(1). Hence there is $\varphi'_{\alpha}\in\mathrm{C}^2(\n,\n)^{\mathrm{R}}$ with $|\alpha|=p+1$ such that 
$d\Phi_{\alpha}=d\mathrm{i}\left(\varphi'_{\alpha}\right),$ 
and 
$\Phi_{\alpha}\in\mathrm{i}\left(\varphi'_{\alpha}\right)+\mathrm{Z}^2\left(\g,\g\right).$ 
Since $\overline{\mathrm{i}}_2$ is sujective, it follows from Corollary \ref{C3.1}(2) that there are $\varphi''_{\alpha}\in\mathrm{C}^2(\n,\n)^{\mathrm{R}}$ with $|\alpha|=p+1$ such that 
$
	\Phi_{\alpha}=\mathrm{i}\left(\varphi'_{\alpha}+\varphi''_{\alpha}\right)+ds_{\alpha},
$
where $s_{\alpha}\in\mathrm{C}^1\left(\g,\g\right)$. We set
$
	s_{p+1}=\mathrm{id}+\sum_{|\alpha|=p+1}t^{\alpha}s_{\alpha},
$
and 
$
	\varphi_\alpha=\varphi'_\alpha+\varphi''_\alpha,
$
for all $|\alpha|=p+1$.
We have
$
	s_{p+1}\ast\Phi =\sum_{|\alpha|\leq p}t^{\alpha}\Phi_{\alpha_{}}+\sum_{|\alpha|=p+1}t^{\alpha_{}}\left(\Phi_{\alpha_{}}-d s_{\alpha_{}}\right)+\mathrm{(degrees\,>p+1)}
	,
$
hence,
$
	s_{p+1}\ast\Phi=\sum_{|\alpha|\leq p+1}t^{\alpha_{}}\mathrm{i}\left(\varphi_{\alpha_{}}\right)+\mathrm{(degrees\,>p+1)}.
$
This deformation satisfies the property $(p+1)$ and the sequence of deformations $\left(s_p\circ\cdots\circ s_1\right)\ast\Phi$ 
converges in the  Krull'sense to a limit under the form $\mathfrak{I}(\varphi)$ which is equivalent to $\phi$. Then $\varphi$ belongs to $\Def(\phi_0,\A)^{\mathrm{R}}$, since $i$ is injective.\\
\textbf{Injection}.
We suppose that $\varphi_1,\varphi_2\in\Def(\varphi_0,\A)^{\mathrm{R}}$ are two deformations of $\varphi_0$ such that their images $\mathfrak{I}(\varphi_1),\mathfrak{I}(\varphi_2)\in\Def(\phi_0,\A)^{\mathrm{R}}$ are equivalent under $\G_m(\A)$. We shall show that there is $\sigma\in\mathfrak{I}(\mathrm{G}^\mathrm{R}_n(\A))$ such that $\mathfrak{I}(\varphi_2)=\sigma\ast\mathfrak{I}(\varphi_1)$ .
We reason by induction on the integer $p\in\N$,
there exists $\sigma=\sum_{\alpha}t^{\alpha }\sigma_{\alpha}\in\G_m(\A)$ such that $\sigma_p:=\sum_{|\alpha|\leq p}t^{\alpha} \sigma_{\alpha}\in\mathfrak{I}(\G^{\mathrm{R}}_n(\m))$ and $\sigma\ast\mathfrak{I}(\varphi_1)=\mathfrak{I}(\varphi_2)$. This is obvious if $p=0$. Then 
$\sigma_p^{-1}\cdot\sigma=\mathrm{id}+\sum_{|\alpha|=p+1}t^{\alpha}\sigma_{\alpha}+\mathrm{(degrees\,>p+1)}.$
Since $\sigma_p\in\mathfrak{I}(\G_n(\A)^{\mathrm{R}})$, $\mathfrak{I}(\varphi_1),\mathfrak{I}(\varphi_2)\in\mathfrak{I}(\Def(\varphi_0,\A)^{\mathrm{R}})$, then the expression,
\begin{eqnarray*} 
  \sigma_p^{-1}\ast(\sigma\ast\mathfrak{I}(\varphi_1))&&=(\sigma_p^{-1}\cdot\sigma)\ast\mathfrak{I}(\varphi_1) =(\mathrm{id}+\sum_{|\alpha|=p+1}t^{\alpha}\sigma_{\alpha}+\cdots)\ast\mathfrak{I}(\varphi_1)\\
 &&=\sum_{|\alpha|\leq p}t^{\alpha}\left(\mathfrak{I}(\varphi_1)\right)_{\alpha}+\sum_{|\alpha|
=p+1}t^{\alpha}\left(\left(\mathfrak{I}(\varphi_1)\right)_{\alpha}-d_{\phi_0}\sigma_{\alpha}\right)\\&&+\mathrm{(degrees\,>p+1)}\\&&  =\sigma_p^{-1}\ast\mathfrak{I}(\varphi_2)
\end{eqnarray*}
belongs to $\mathfrak{I}(\Def(\varphi_0,\A)^{\mathrm{R}})$. From Lemma \ref{n33}, we have   
$d_{\phi_0}\sigma_{\alpha}\in\mathrm{i}_2(\mathrm{Z}^2(\n,\n)^{\mathrm{R}}),$ 
for all $|\alpha|=p+1$. Since $\overline{\mathrm{i}}_2$ is injective, then  $d_{\phi_0}\sigma_{\alpha}\in\mathrm{i}_2(\mathrm{B}^2(\n,\n)^{\mathrm{R}})$ by Corollary \ref{C3.1}(1). Hence there are $s_{\alpha}\in\mathrm{C}^1(\n,\n)^{\mathrm{R}}$ such that
$d_{\phi_0}\sigma_{\alpha}=d_{\phi_0}\mathrm{i}_1(s_{\alpha})$
and 
	$\sigma_{\alpha}\in\mathrm{i}_1(s_{\alpha})+\mathrm{Z}^1(\g,\g),
$
for all $|\alpha|=p+1$. Since $\overline{\mathrm{i}}_1$ is surjective, it follows from Corollary \ref{C3.1}(2) that there are $\delta^1_{\alpha}\in\mathrm{ad}\g=\mathrm{B}^1(\g,\g)$ and $\delta^2_{\alpha}\in\mathrm{C}^1(\n,\n)^{\mathrm{R}}$ such that
$\sigma_{\alpha}=\mathrm{i}_1(s_{\alpha})+\delta^1_{\alpha}+\mathrm{i}_1(\delta^2_{\alpha})$
for all $|\alpha|=p+1$. The derivations $\delta^1_{\alpha}$ being inner, it follows that they may be lifted to derivations $\delta^1_{\alpha}(t)$ of $\mathfrak{I}(\varphi_1)$ by Lemma \ref{n6}. Since $\Pi:=\exp(-\sum_{|\alpha|=p+1}t^{\alpha}\delta^1_{\alpha}(t))$ is an automorphism of $\mathfrak{I}(\varphi_1)$ by Lemma \ref{n7} and the induction hypothesis, hence we get
$$\sigma_{p+1}:=\sigma_{p}\circ\Pi=\sigma_p+\sum_{|\alpha|=p+1}t^{\alpha}\mathrm{i}_1(s_{\alpha}+\delta^2_{\alpha})+\mathrm{(term\,of\,degrees\,>p+1)}.$$
Then $\sigma\circ\Pi$ verifies the property for $p+1$. The sequence $(\sigma_p)$ converges in the sense of Krull to an automorphism $\sigma_\infty$ which belongs to $\mathfrak{I}(\G_n^{\mathrm{R}}(\A))$ satisfying $\sigma_\infty\ast\mathfrak{I}(\varphi_1)=\mathfrak{I}(\varphi_2)$.\\
2. \textbf{Local Isomorphism $\overline{\eta}$.}
Let $f:\mathrm{B}\rightarrow\mathrm{C}$ be a local $\K$-algebra morphism and $\widehat{f}:\widehat{\mathrm{B}}\rightarrow\widehat{\mathrm{C}}$ its unique extension on the completion $\K$-algebras. 
Set $\A=\widehat{\O^{\mathcal{A}}_{\phi_0}}$. From the surjectivity on the classes, there are an element $h:\O^{\mathrm{R}}_{\varphi_0}\rightarrow\widehat{\O^{\mathcal{A}}_{\phi_0}}$ of $\Def^{\mathrm{R}}(\varphi_0,\widehat{\O^{\mathcal{A}}_{\phi_0}})$ and an element $s$ of $\G_m(\widehat{\O^{\mathcal{A}}_{\phi_0}})$ such that  
\begin{equation}\label{E3.15}
	\widehat{h}\circ\widehat{\eta}=s\ast\widehat{\pi}.
\end{equation}
By Corollary \ref{N21} and Remark \ref{r2.1}, there is a local morphism $\overline{h}:\O^{\mathrm{R},\mathcal{A}'}_{\varphi_0}\rightarrow\widehat{\O^{\mathcal{A}}_{\phi_0}}$ and $s'\in\G_n^{\mathrm{R}}(\widehat{\O^{\mathcal{A}}_{\phi_0}})$ such that 
\begin{equation}\label{E3.16}
\widehat{\overline{h}}\circ\widehat{\pi'}=s'\ast \widehat{h}.
\end{equation}
Since $\widehat{\pi}$ is an epimorphism (because is $\pi$), it follows from the equation (\ref{E3.15}) that $\widehat{h}$ is an epimorphism by Lemma \ref{L}. From the equation (\ref{E3.16}), we deduce that $\widehat{\overline{h}}$ is an epimorphism. Consequently, $\widehat{\overline{h}}\circ\widehat{\overline{\eta}}:\widehat{\O^{\mathcal{A}}_{\phi_0}}\rightarrow\widehat{\O^{\mathcal{A}}_{\phi_0}}$ is an epimorphism ($\widehat{\overline{\eta}}$ is an epimorphism, since $\mathrm{i}_2$ is injective). Since $\widehat{\O^{\mathcal{A}}_{\phi_0}}$ is the inverse limit of $(\O^{\mathcal{A}}_{\phi_0}/\m^n(\O^{\mathcal{A}}_{\phi_0}))_n$, it follows that the homomorphisms $(\widehat{\overline{h}}\circ\widehat{\overline{\eta}})_n:\O^{\mathcal{A}}_{\phi_0}/\m^n(\O^{\mathcal{A}}_{\phi_0})\rightarrow\O^{\mathcal{A}}_{\phi_0}/\m^n(\O^{\mathcal{A}}_{\phi_0})$ induced by $\widehat{\overline{h}}\circ\widehat{\overline{\eta}}$ are isomorphisms $(n=1,2,\cdots)$, since $\widehat{\overline{h}}\circ\widehat{\overline{\eta}}$ is surjective and $\O^{\mathcal{A}}_{\phi_0}/\m^n(\O^{\mathcal{A}}_{\phi_0})$ is a $\O^{\mathcal{A}}_{\phi_0}/\m(\O^{\mathcal{A}}_{\phi_0})$-module of finite dimension (since $\O^{\mathcal{A}}_{\phi_0}$ is Noetherian). Hence $\widehat{\overline{h}}\circ\widehat{\overline{\eta}}$ becomes an isomorphism, and then $\widehat{\overline{\eta}}$ is injective. Then $\widehat{\overline{\eta}}$ is an isomorphism (since $\widehat{\overline{\eta}}$ is an epimorphism). We deduce that $\widehat{\overline{\eta}}$ is an isomorphism if and only if is $\overline{\eta}$, by Corollary 1.6, p.282 \cite{A}, since $\K=\C$.
\end{proof}
\begin{prop}If $\mathcal{A}_1$ and $\mathcal{A}_2$ are two admissible sets at $\phi_0$, then the local rings $\O_{\phi_0}^{\mathcal{A}_1}$  and $\O_{\phi_0}^{\mathcal{A}_2}$ are isomorphic, for $\K=\C$.
\end{prop}
\begin{proof} Since $\pi_i:\O_{\phi_0}\rightarrow\O_{\phi_0}^{\mathcal{A}_i}$ is versal, for $i=1,2$, there are two local morphisms $\overline{\pi}_1:\O_{\phi_0}^{\mathcal{A}_2}\rightarrow\O_{\phi_0}^{\mathcal{A}_1}$ and $\overline{\pi}_2:\O_{\phi_0}^{\mathcal{A}_1}\rightarrow\O_{\phi_0}^{\mathcal{A}_2}$ such that $\widehat{\pi}_1$ (resp. $\widehat{\pi}_2$) is equivalent to $\widehat{\overline{\pi}_1}\circ\widehat{\pi_2}$ (resp.  $\widehat{\overline{\pi}_2}\circ\widehat{\pi_1}$). It follows from above and  since $\pi_i$ is surjective that $\overline{\pi}_i$ is sujective, for $=1,2$. The local morphism $\widehat{\overline{\pi}}_1\circ\widehat{\overline{\pi}}_2:\widehat{\O_{\phi_0}^{\mathcal{A}_1}}\rightarrow\widehat{\O_{\phi_0}^{\mathcal{A}_1}}$ and $\widehat{\overline{\pi}}_2\circ\widehat{\overline{\pi}}_1:\widehat{\O_{\phi_0}^{\mathcal{A}_2}}\rightarrow\widehat{\O_{\phi_0}^{\mathcal{A}_2}}$ are surjective as composition of surjective maps. Since $\O_{\phi_0}^{\mathcal{A}_i}$ is Noetherian, then $\widehat{\overline{\pi}_1}\circ\widehat{\overline{\pi}_2}$ and $\widehat{\overline{\pi}_2}\circ\widehat{\overline{\pi}_1}$ are bijective, by using the same argument as in the proof of Theorem \ref{N31}. Hence the map $\widehat{\overline{\pi}}_i$ is injective, so is bijective, for $1,2$. We deduce that $\overline{\pi}_i$ is bijective by Corollary 1.6, p.282 \cite{A}.
\end{proof}
\begin{prop}\label{n3.2}An algebraic Lie algebra $\g=\mathrm{R}\ltimes\n$ with reductive part $\mathrm{R}=\mathrm{U}\oplus\mathrm{S}$ where $\mathrm{S}$ is the semi-simple part and $\mathrm{U}$ is the torus part, satisfies the hypotheses of Theorem \ref{N31} iff it belongs to one of the following cases
\begin{enumerate}
	\item $\mathrm{U}\neq 0$, $\g$ is complete, $\mathrm{H}^1\left(\n,\g/\n\right)^\mathrm{R}=0=\mathrm{H}^2(\n,\g/\n)^{\mathrm{R}}$.
\item $\mathrm{U}=0$, $\mathrm{H}^1\left(\n,\g/\n\right)^\mathrm{R}=0=\mathrm{H}^2(\n,\g/\n)^{\mathrm{R}}$.
\end{enumerate}
\end{prop}
\begin{proof} With the long sequence of cohomology of (\ref{C}) and the factorization of Hochschild-Serre \cite{HS}, we can show the statement.
\end{proof}

Let $\n$ be a Lie algebra, a derivation $\delta$ of $\n$ has a strict positive spectrum, denoted by $\delta>0$, if all its eigenvalues are in $\Q^+-\left\{0\right\}$. The algebraic Lie algebra $\Der_\K(\n)$ admits a Chevalley's decomposition $\mathrm{S}\oplus\mathrm{A}\oplus\mathrm{N}$ where $\mathrm{S}$ is a Levi subalgebra, $\mathrm{A}$ is an abelian subalgebra such that $\left[\mathrm{S},\mathrm{A}\right]=0$, $\mathrm{N}$ is the largest ideal of nilpotence and $\mathrm{R}=\mathrm{S}\oplus\mathrm{A}$ the reductive part.
\begin{prop}\label{n3.3}If there is a derivation of $\n$ such $\delta>0$ then $\n$ is nilpotent and we have:
\begin{enumerate}
	\item $\mathrm{A}$ is nonzero and contains a derivation with strict positive spectrum.
	\item Any complete Lie algebra with nilpotent radical $\n$ is given by $\g_1=\mathrm{R}_1\ltimes\n$ with $\mathrm{R}_1=\mathrm{S}_1\oplus\mathrm{T}_1\subset \mathrm{R}$ such that $\mathrm{S}_1$ is a semisimple subalgebra of $\mathrm{S}$ and $\mathrm{T}_1$ is a maximal torus in $\mathrm{Der}(\n)^{\mathrm{S}_1}=\mathrm{S}^{\mathrm{S}_1}+\mathrm{A}+ \mathrm{N}^{\mathrm{S}_1}$ with $\mathrm{N}^{\mathrm{R}_1}=0$. There is a solvable Lie algebra of this type $\mathrm{T}\ltimes\n$ if $\mathrm{N}^{\mathrm{T}}=0$ for a maximal torus $\mathrm{T}$ on $\n$. The Lie algebra $\g_1$ satisfies the hypotheses of the theorem of reduction.
\end{enumerate}
\end{prop}
\subsection{Formal Rigidity}
Let $\phi_0\in\L_m(\K)$ and $\mathcal{A}$ an admissible set at $\phi_0$. The Krull's dimension $d$ of the completion local ring $\widehat{{\O}^\mathcal{A}_{\phi_0}}$ of $\O^\mathcal{A}_{\phi_0}$ for the $\m(\O^\mathcal{A}_{\phi_0})$-adic topology is the maximal number of elements $t_1,...,t_d$ of $\m(\O^\mathcal{A}_{\phi_0})$ such that the subring of formal power series in $t_1,...,t_d$, is isomorphic to the formal power series ring $\K[[T_1,...,T_d]]$ in $d$ variables. This gives the dimension of $\L^{\mathcal{A}}_{m,\phi_0}$ at point $\phi_0$. The formal rigidity is the rigidity relative to the base $\K\left[\left[T\right]\right]$.
If we suppose that there is a valuation on $\K$, for example $\C$, we can define a separated strong topology on $\K^m$ and use the notion of convergence of series (analyticity). In this case the formal rigidity is equivalent to the analytic rigidity, by M. Artin's theorem, \cite{A},\cite{C7}. This rigidity is also equivalent to the orbit is open in the senses of Zariski and strong topology, cf. \cite{NR}. We call it the geometric rigidity at point $\phi_0$.
\begin{thm}\label{t4.1}For all admissible set $\mathcal{A}$ of $\mathcal{I}$ at $\phi_0$, then the following conditions are equivalent
\begin{enumerate}
\item $\dim_\K\O_{\phi_0}^\mathcal{A}<\infty$;
\item the Krull's dimension $d$ of $\widehat{\O_{\phi_0}^\mathcal{A}}$ is null; $\phi_0$ is an isolated point of $\L_m^{\phi_0,\mathcal{A}}(\K)$;
\item the elements of the maximal ideal of $\O_{\phi_0}^\mathcal{A}$ are nilpotent;
\item $\phi_0$ is formally rigid in $\L_m$;
\item Moreover if $\K$ is a valued field, the orbit $\left[\phi_0\right]$ is a Zariski's open set in $\L_m(\K).$
\end{enumerate}
\end{thm}
\begin{proof}$1\Rightarrow 2.$ If $t$ is an element of $\m(\O^{\mathcal{A}}_{\phi_0})$ then the sequence $(t^i)_{i\in\N}$ linearly generates a vector subspace of finite dimension. Consequently, the subring $\K[[t]]$ of $\widehat{\O^{\mathcal{A}}_{\phi_0}}$ cannot be a ring of series power, so $d=0$.\\
$2\Rightarrow 3.$ If $t\in\m(\O^{\mathcal{A}}_{\phi_0})$ then the ring $\K[[t]]$ is isomorphic to the ring $\K[T]/(T^p)$ where $p\in\N^*$, and $t^p=0$.\\
The equivalences $1,3,4,5$ are proved in Proposition 6.6 \cite{C9}.
\end{proof}\\
A formal rigid law in $\L_m$ is algebraic \cite{C1} and admits a Chevalley's decomposition $\g=\mathrm{R}\ltimes\n$.
\begin{cor}\label{c4.1}
\begin{enumerate}
\item If $\g=\mathrm{R}\ltimes\n$ is such that $[\g,\n]=\n$ and $\mathrm{H}^2(\n,\g/\n)^{\mathrm{R}}=0$, then $\g$ is formal rigid in $\L_m$ if and only if is $\n$ in $\L_n^{\mathrm{R}}$.
	\item If $\g=\mathrm{R}\ltimes\n$ is such that $[\g,\n]=\n$ and $\g$ is formal rigid in $\L_m$, then $\n$ is formal rigid in $\L_n^{\mathrm{R}}$.
\end{enumerate}
\end{cor}
\begin{proof} 1. Let $\phi_0$ (resp. $\varphi_0$) denote the law of $\g$ (resp. $\n$). For all admissible set $\mathcal{A}\subset\mathcal{I}$ (resp. $\mathcal{A'}\subset\mathrm{I'}$) at $\phi_0$ (resp. $\varphi_0$), defined as in Theorem \ref{N31}, the local morphism $\overline{\eta}:\O^{\mathcal{A}}_{\phi_0}\rightarrow\O^{\mathrm{R},\mathcal{A'}}_{\varphi_0}$ is an isomorphism since $[\g,\n]=\n$ and $\mathrm{H}^2(\n,\g/\n)^{\mathrm{R}}=0$, it follows that the dimension of the space $\O^{\mathcal{A}}_{\phi_0}$ is finite if and only if it is for $\O^{\mathrm{R},\mathcal{A'}}_{\varphi_0}$. Then we deduce the statement from Theorem \ref{t4.1}.\\
2. $\mathrm{\eta}$ is surjective since $[\g,\n]=\n$, we deduce the statement from Theorem \ref{t4.1}
\end{proof}
\begin{cor}If $\g_k=\mathrm{R}_k\ltimes\n_k$ satisfies $\mathrm{H}^1(\n_k,\g_k/\n_k)^{\mathrm{R}_k}=\mathrm{H}^2(\n_k,\g_k/\n_k)^{\mathrm{R}_k}=0$ and $\left[\mathrm{R}_k,\n_k\right]=\n_k$ for $k=1,2$, then $\g_1\times\g_2$ is formal rigid if and only if $\g_1$ and $\g_2$ are formal rigid.
\end{cor}
\begin{proof} If $\g,\n$ and $\mathrm{R}$ are the direct products of $\g_k,\n_k$ and $\mathrm{R}_k$ respectively, we then obtain the same conditions in the statement for $\g$. In fact, $\mathrm{H}^2(\n,\g/\n)^{\mathrm{R}}$ is equal to $\mathrm{Hom}_{\mathrm{R}}(\mathrm{H}_2(\n),\g/\n)$ where the second homology group $\mathrm{H}_2(\n)$ is isomorphic to $\mathrm{H}_2(\n_1)\oplus\mathrm{H}_2(\n_2)\oplus\mathrm{E}$ with $\mathrm{E}=(\n_1/\left[\n_1,\n_1\right])\wedge\n_2\oplus\n_1\wedge(\n_2/\left[\n_2,\n_2\right])$; we have $\mathrm{Hom}_{\mathrm{R}}(\mathrm{E},\g/\n)=0$. 
The condition $\left[\mathrm{R},\n\right]=\n$ implies that $\L_n^{\mathrm{R}}=\L_{n_1}^{\mathrm{R}_1}\times\L_{n_2}^{\mathrm{R}_2}$. Since $\mathrm{B}^2(\n,\n)^{\mathrm{R}}=\mathrm{B}^2(\n_1,\n_1)^{\mathrm{R}_1}\oplus\mathrm{B}^2(\n_2,\n_2)^{\mathrm{R}_2}$, then $\mathcal{A}=\mathcal{A}_1\times\mathcal{A}_2$ is an admissible set at $\n$, with $\mathcal{A}_i$ admissible sets at $\n_i$ for $i=1,2$. Hence $\L_n^{\mathrm{R},\mathcal{A}}=\L_{n_1}^{\mathrm{R}_1,\mathcal{A}_1}\times\L_{n_2}^{\mathrm{R}_2,\mathcal{A}_2}$, and the local ring $\O_\n^{\mathrm{R},\mathcal{A}}$ of $\L_{n}^{\mathrm{R},\mathcal{A}}$ at $\n$ is isomorphic to $\O_{\n_1}^{\mathrm{R}_1,\mathcal{A}_1}\otimes_\K\O_{\n_2}^{\mathrm{R}_2,\mathcal{A}_2}$ with $\O_{\n_k}^{\mathrm{R}_k}$ the local ring of $\L_{n_k}^{\mathrm{R}_k}$ at $\n_k$, $k=1,2$. Then the $\K$-dimension of $\O_\n^{\mathrm{R},\mathcal{A}}$ is finite iff it is for $\O_{\n_1}^{\mathrm{R}_1,\mathcal{A}_1}$ and $\O_{\n_2}^{\mathrm{R}_2,\mathcal{A}_2}$. We deduce the result from Theorem \ref{t4.1} and Corollary \ref{c4.1}.
\end{proof}
\section{Versality in Central Extensions}
We shall construct a sequence of central extensions of schemes $(\L_n^{\mathrm{T}})_n$ for $n\geq n_0$ under some hypotheses on the weights of $\mathrm{T}$ such that \\
$1)$ any law  $\varphi_{n+1}\in\L_{n+1,\varphi_0}^{\mathrm{T}}$ is a central extension of a law $\varphi_n\in\L_{n,\varphi_0}^{\mathrm{T}}$, where $\mathrm{T}$ is a maximal torus of $\varphi_n$, which is extended to $\varphi_{n+1}$ by adding a weight $\alpha_{n+1}$;\\
$2)$ the construction of a versal deformation of $\mathrm{T}\ltimes\varphi_n$ in $\L_m$ is equivalent to that of $\varphi_n$ in $\L_n^{\mathrm{T}}$, via the reduction theorem;\\
$3)$ a versal deformation of the laws $\varphi_{n+1}$ is deduced from that of $\varphi_n$ by central extension starting from an initialization $n=n_0$, where the torus $\mathrm{T}$ appears.
\subsection{Central Extensions of Schemes}
A torus $\mathrm{T}$ on $\K^n$ is defined by  the set $\pi_n$ of its weights $\alpha_i\in\mathrm{T}^*$,\\$t\cdot e_i=\alpha_i(t)e_i, \left(1\leq i\leq n\right), t\in\mathrm{T}.$
We denote again by $\mathrm{T}$ the torus on $\K^{n+1}$ defined by adding the weight $\alpha_{n+1}\in\mathrm{T}^*$:
$t\cdot e_{n+1}=\alpha_{n+1}(t)e_{n+1},\quad t\in\mathrm{T}.$\\
The variety $\L_n^{\mathrm{T}}(\K)$ is the set of Lie multiplications $\varphi$ defined by \\
$\varphi\left(e_i,e_j\right)=\sum_{k=1}^n\varphi^k_{ij}e_k$ such that $
	(\alpha_k-\alpha_i-\alpha_j)(t)\varphi^k_{ij}=0, t\in\mathrm{T}, i<j,k$.
This is equivalent to $\varphi^k_{ij}=0$ for $\alpha_k\neq\alpha_i+\alpha_j$.
Denote by $\mathcal{J}$ the set the multi-indices $\left(^k_{ij}\right)$ such that $i<j$ and $\alpha_i+\alpha_j=\alpha_k$. The coordinates $X^k_{ij}$ in $\L_n^{\mathrm{T}}$ are indexed by $\mathcal{J}$, and the Jacobi's relations are given by
\begin{equation}\label{e43}
	\mathrm{J}^h_{ijk}=\oint_{(ijk)}\sum_{l=1}^nX^l_{ij}X^h_{lk}=0,\quad 1\leq i<j<k\leq n,\quad 1\leq h\leq n,
\end{equation}
for $\alpha_i+\alpha_j+\alpha_k=\alpha_h\in\pi_n.$\\ 
Let $\sum_n\left(\mathrm{T}\right)$ be the subset of $\L_n^{\mathrm{T}}\left(\K\right)$ consisting of laws such that $\mathrm{T}$ is exactly a maximal torus of derivations. If $\sum_n\left(\mathrm{T}\right)$ is nonempty then $\mathrm{T}$ is algebraic.
Let $\mathrm{T}$ be the torus on $\K^{n+1}$ given by $\pi_{n+1}=\pi_{n}\cup\left\{\alpha_{n+1}\right\}$. We suppose the choice of $\alpha_{n+1}$ such that $\sum_{n+1}\left(\mathrm{T}\right)$ is nonempty. The weight of nonzero weightvector of the form $\left[e_i,e_{n+1}\right]$ for $i\leq n$ is $\alpha_i+\alpha_{n+1}\in\pi_{n+1}$. If $0\notin\pi_{n}$ and $\alpha_{n+1}\notin\pi_{n}-\pi_{n}$ then $\left[e_i,e_{n+1}\right]=0$ for $i\leq n$ and $\K e_{n+1}$ is central. This means that any law $\varphi_{n+1}\in\L^{\mathrm{T}}_{n+1}\left(\K\right)$ is a central extension of a law $\varphi_n\in\L^{\mathrm{T}}_{n}\left(\K\right)$ of kernel $\K e_{n+1}$, i.e. we have the exact sequence of Lie algebras
\begin{equation}\label{e44}
	\left\{0\right\}\longrightarrow\K e_{n+1}\longrightarrow\left(\K^{n+1},\varphi_{n+1}\right)\longrightarrow\left(\K^{n},\varphi_{n}\right).
\end{equation}
If $\mathrm{T}$ is maximal on $\left(\K^{n+1},\varphi_{n+1}\right)$, i.e.  $\varphi_{n+1}\in\sum_{n+1}\left(\mathrm{T}\right)$, then the extension is not trivial and corresponds to a nonzero class of $\mathrm{H}^2\left(\varphi_n,\K e_{n+1}\right)^{\mathrm{T}}$, for the trivial action. The group of homology $\mathrm{H}_2\left(\varphi_n\right)$ is a $\mathrm{T}$-module decomposed as
\begin{equation}\label{e45}
\mathrm{H}_2\left(\varphi_n\right)=\bigoplus_{\alpha\in\mathrm{T}^*}\mathrm{H}_2\left(\varphi_n\right)_{\alpha}.
\end{equation}
Thus
\begin{equation}\label{e46}
	\mathrm{H}^2\left(\varphi_n,\K e_{n+1}\right)^{\mathrm{T}}=\mathrm{Hom}_{\mathrm{T}}\left(\mathrm{H}_2\left(\varphi_n\right),\K e_{n+1}\right)\cong\left(\mathrm{H}_2(\varphi_n\right)_{\alpha_{n+1}})^*.
\end{equation}
It follows from (\ref{e46}) that a nonzero cohomology class  corresponds to a nonzero homology class such that the weight $\alpha_{n+1}$ appears in the decomposition (\ref{e45}) and implies that $
	\alpha_{n+1}=\alpha_i+\alpha_j\quad \mathrm{for}\quad i<j\leq n.$
We suppose that $\pi_n$ consists of strict positive weights, i.e., there is $t\in\mathrm{T}$ such that each weight  $\alpha_i$ satisfies $\alpha_i(t)\in\Q^{*+}$ for $1\leq i\leq n$. We deduce the following properties:
\begin{enumerate}
		\item $\L_n^{\mathrm{T}}\left(\K\right)$ is formed of nilpotent Lie multiplications,
		\item $\pi_{n+1}$ is formed of strict positive weights,
		\item $\mathrm{T}\ltimes\varphi_n$ where $\varphi_n\in\L^{\mathrm{T}}_n$ is complete and satisfies the reduction theorem. 
\end{enumerate}
\begin{de}A sequence of weights, $\pi=\left(\alpha_p\right)\subset\mathrm{T}^*$, is a path of weights if
\begin{enumerate}
	\item there is $n_0\in\N^*$ such that the family $\left(\alpha_1,...,\alpha_{n_0}\right)$ generates $\mathrm{T}^*$ over $\K$ and $\alpha_i>0$ for all $i=1,...,n_0$,
	\item $\sum_{n}\left(\mathrm{T}\right)$ is a nonempty set for all $n\geq n_0$,
	\item $\alpha_{n+1}$ does not belong to $(\pi_n-\pi_n)$ for all $n\geq n_0$.
\end{enumerate}
$n_0$ is called the initialization of $\pi$, where the torus $\mathrm{T}$ is maximal.
\end{de}
The initialization $n_0$ of $\pi$ corresponds to the value of $n$ where the torus $\mathrm{T}$ is maximal. 
The condition of the central extension will give the existence of a morphism of schemes which generalizes the quotient map (\ref{e44}). Let $\mathrm{P}_n$ denote the polynomial ring $\K[X^k_{ij}:\binom{k}{ij}\in\mathcal{J}]$ and $\mathrm{J}_n$ the ideal generated by the Jacobi's polynomials $\mathrm{J}^h_{ijk}$, see (\ref{e43}). 
\begin{prop}\label{n41}If $\left(\alpha_p\right)_{p\in\N}\subset\mathrm{T}^*$ is a path of weights with $n\geq n_0$, then 
\begin{enumerate}
	\item the ideal $\mathrm{J}_{n+1}$ is generated by the ideal $\mathrm{J}_n$ and the polynomials $\mathrm{J}^{n+1}_{ijk}$ where $1\leq i<j<k\leq n$;
	\item the canonical monomorphism $\mathfrak{i}_n:\mathrm{P}_n\rightarrow\mathrm{P}_{n+1}$ induces a morphism on the quotients $\mathrm{P}_n/\mathrm{J}_n\rightarrow\mathrm{P}_{n+1}/\mathrm{J}_{n+1}$ and a scheme morphism $\mathfrak{p}_{n+1}:\L^\mathrm{T}_{n+1}\rightarrow\L^\mathrm{T}_{n}$ 
defined on the space of rational points $\L^{\mathrm{T}}_{n+1}(\K)$ by the quotient map (\ref{e44}).
\end{enumerate}
\end{prop}
\begin{proof} 1. Consider  a nontrivial Jacobi's polynomial $\mathrm{J}^h_{ijk}$ with $h\leq n$; we will show that it belongs to $\mathrm{J}_n$. It is clear for $k\leq n$. If $k=n+1$ then $\alpha_i+\alpha_j+\alpha_{n+1}\in\pi_{n+1}$. It follows by assumptions that one of the following elements $\alpha_{n+1}+\alpha_i$, $\alpha_{n+1}+\alpha_j$, $\alpha_{n+1}+\alpha_i+\alpha_j$ belongs to $\pi_{n+1}$ with $\alpha_i+\alpha_j\in\pi_{n+1}$. However  if  $\alpha_{i}+\alpha_j\in\pi_n$ then the elements $\alpha_{n+1}+\alpha_i$, $\alpha_{n+1}+\alpha_j$ and $\alpha_{n+1}+\alpha_i+\alpha_j$  do not belong to $\pi_{n+1}$ by hypothesis on $\pi_{n+1}$. If $\alpha_{i}+\alpha_j\notin\pi_n$ then $\alpha_{i}+\alpha_j=\alpha_{n+1}$ and its multiplicity is $1$. The weight $2\alpha_{n+1}$ corresponds to a trivial bracket. By similar arguments we show that for $i<j<k$ the polynomial $\mathrm{J}^{n+1}_{ijk}$ satisfies  $k\leq n$. We deduce that $\mathrm{J}_{n+1}$ is generated by $\mathrm{J}_n$ and the polynomials $\mathrm{J}^{n+1}_{ijk}$. \\
2. The correspondence $X^k_{ij}\rightarrow X^k_{ij}$ for $i<j\leq n$ and $k\leq n$ induces an injective map $\mathrm{i}_n:\mathrm{P}_n\rightarrow\mathrm{P}_{n+1}$ which sends $\mathrm{J}_n$ to $\mathrm{J}_{n+1}$. It induces a quotient morphism $\overline{\mathfrak{i}}_n:\mathrm{P}_{n}/\mathrm{J}_{n}\rightarrow \mathrm{P}_{n+1}/\mathrm{J}_{n+1}$ and $\mathfrak{p}_{n+1}:=\mathrm{Spec}(\overline{\mathfrak{i}}_n):\L^\mathrm{T}_{n+1}\rightarrow\L^\mathrm{T}_{n}$.
\end{proof}
\subsection{Versality in Central Extensions}
The laws of the open set $\sum_{n+1}(\mathrm{T})$ are coming from $\sum_{n}(\mathrm{T})$ in general, and it is even possible that there are laws $\varphi_n$ of $\sum_{n}(\mathrm{T})$ no having an extension in $\sum_{n+1}(\mathrm{T})$, this happens when $\mathrm{H}_2(\varphi_n)_{\alpha_{n+1}}$ is zero. However, an extension $\varphi_{n+1}$ of $\varphi_n\in\sum_{n}(\mathrm{T})$, cf (\ref{e44}), belongs to $\sum_{n+1}(\mathrm{T})$ iff it is not trivial. We will say that $\varphi_{n+1}$ is obtained by a direct filiation of $\varphi_n$. This means that there are two admissible sets $\mathcal{A}_n$ and $\mathcal{A}_{n+1}$ at $\varphi_{n}$ and $\varphi_{n+1}$ respectively such that $ \mathcal{A}_{n+1}=\mathcal{A}_{n}\cup\left\{(^{n+1}_{pq})\right\}$ 
for some $p<q$ satisfying $(\varphi_{n+1})^{n+1}_{pq}\neq 0$. 
The inclusion $\mathcal{A}_{n}\hookrightarrow\mathcal{A}_{n+1}$ thus induces a process of construction by central extension (\ref{e44}).
Such an obtained sequence of schemes is called a direct filiation; next we shall develop it.
The trace $\mathfrak{p}^{-1}_{n+1}(\varphi_n)\cap\L_{n+1}^{\mathrm{T},\mathcal{A}_{n+1}}$ of the fiber of $\mathfrak{p}_{n+1}:\L_{n+1}^{\mathrm{T}}\rightarrow\L_n^{\mathrm{T}},$ 
of each point $\varphi_n\in\L_n^{\mathrm{T},\mathcal{A}_n}$ on $\L_{n+1}^{\mathrm{T},\mathcal{A}_{n+1}}$ is the set of laws $\varphi_{n+1}\in\L_{n+1}^{\mathrm{T},\mathcal{A}_{n+1}}$ which satisfy (\ref{e44}) and the condition $(\varphi_{n+1})^{n+1}_{pq}=1$. As subscheme, it is the set of $(X^{n+1}_{ij})$ which verify 
\begin{equation}\label{e52}
	X_{pq}^{n+1}=1,
\end{equation}
and Jacobi's equations
\begin{equation}\label{e51}
	\oint_{ijk}\sum_l(\varphi_n)^l_{ij} X_{lk}^{n+1}=0. 
\end{equation}

Let $\n$ be a nilpotent Lie algebra with bracket $\varphi$, and $\mathfrak{a}$ a central ideal of dimension one which is stable under a torus $\mathrm{T}$ on $\n$. Let $\mathfrak{b}$ be a complement of $\mathfrak{a}$ in $\n$ stable under $\mathrm{T}$, identified with $\n/\mathfrak{a}$ as $\mathrm{T}$-module. Hence $\varphi$ may be written as $\varphi_0+\psi$ where $\varphi_0$ is the Lie bracket defined on $\n/\mathfrak{a}$ transferred on $\mathfrak{b}$, and $\psi$ is a cocycle of $\mathrm{Z}^2(\n/\mathfrak{a},\mathfrak{a})^{\mathrm{T}}$.
Let $\Omega(\n)$ be the $\mathrm{T}$-submodule of $\wedge^3\n$ generated by the vectors $\oint_{(xyz)}\varphi(x,y)\wedge z$. 
If $\beta$ is the weight of $\mathrm{T}$ on $\mathfrak{a}$, then the group $\mathrm{H}^2(\n,\mathfrak{a})^{\mathrm{T}}$ is isomorphic to 
$\mathrm{Hom}_{\mathrm{T}}(\mathrm{H}_2(\n),\mathfrak{a})=(\mathrm{H}_2(\n)_{\beta})^*.$
It follows from Proposition 3.2 in \cite{C2} that $\mathrm{H}_2(\n)_{\beta}$ is isomorphic to 
$
	\frac{(\ker\widetilde{\varphi})_{\beta}}{\Omega(\n)_{\beta}}.
$
\begin{prop}\label{nn52} If $\mathrm{T}$ does not have a null weight on $\n/\mathfrak{a}$, then the vector space $\mathrm{H}_2(\n)_{\beta}$ is isomorphic to $\frac{(\ker\widetilde{\varphi}_0)_{\beta}\cap(\ker\widetilde{\psi})_{\beta}}{\Omega(\n/\mathfrak{a})_{\beta}}.$ 
If the class of $\psi$ is nonzero, then it is a hyperplan of $\mathrm{H}_2(\n/\mathfrak{a})_{\beta}$.
\end{prop}
\begin{proof} We have
$$
	\widetilde{\varphi}((b_1+a_1)\wedge (b_2+a_2))=\varphi(b_1+a_1,b_2+a_2)=\varphi(b_1,b_2)=\varphi_0(b_1,b_2)+\psi(b_1,b_2),$$
for all $ (b_1,b_2,a_1,a_2)\in\mathfrak{b}^2\times\mathfrak{a}^2$. It follows that
\begin{equation}\label{e5.1}
(\ker\widetilde{\varphi})_{\beta}=(\ker\widetilde{\varphi}_0)_{\beta}\cap(\ker\widetilde{\psi})_{\beta}\oplus (\n\wedge\mathfrak{a})_{\beta}.
\end{equation}
Hence $(\n\wedge\mathfrak{a})_{\beta}=0$ since $\mathfrak{a}\wedge\mathfrak{a}=0$ and $\mathfrak{b}\simeq\n/\mathfrak{a}$ does not admit a null weight by hypothesis. By Eq. (\ref{e5.1}), we have $(\ker\widetilde{\varphi})_{\beta}=(\ker\widetilde{\varphi}_0)_{\beta}\cap(\ker\widetilde{\psi})_{\beta}$.
The space $\Omega(\n)_{\beta}$ is generated by the cyclic sums of tensors and we have
$$\varphi(b_1+a_1,b_2+a_2)\wedge(b_3+a_3)=\varphi_0(b_1,b_2)\wedge b_3+\varphi_0(b_1,b_2)\wedge a_3+\psi(b_1,b_2)\wedge b_3,$$
for all $(b_1,b_2,b_3,a_1,a_2,a_3)\in\mathfrak{b}^3\times\mathfrak{a}^3$.
According to the hypothesis, the last two tensors have a null projection on the weightspace of $\beta$ and hence there are isomorphisms
$\Omega(\n)_{\beta}\simeq\Omega(\n/\mathfrak{a})_{\beta}\quad\mathrm{and}\quad \mathrm{H}_2(\n)_{\beta}\simeq((\ker\widetilde{\varphi}_0)_{\beta}\cap(\ker\widetilde{\psi})_{\beta})/\Omega(\n/\mathfrak{a})_{\beta}.$
The homomorphism $\widetilde{\psi}:\wedge^2\n/\mathfrak{a}\rightarrow\mathfrak{a}$ is zero on a complement of $(\wedge^2\n/\mathfrak{a})_{\beta}$. If it is zero on $(\ker\widetilde{\varphi}_0)_\beta$, i.e. on $\ker\widetilde{\varphi}_0$, then it factorizes through the quotient $\wedge^2(\n/\mathfrak{a})/(\ker\widetilde{\varphi}_0)$. Hence $\wedge^2(\n/\mathfrak{a})/(\ker\widetilde{\varphi}_0)$ and $\left[\n/\mathfrak{a},\n/\mathfrak{a}\right]$ are isomorphic and $\widetilde{\psi}$ may be written as $h\circ\widetilde{\varphi}_0$ for some $h\in\mathrm{Hom}_\K(\left[\n/\mathfrak{a},\n/\mathfrak{a}\right],\mathfrak{a})$ , thus the class of $\psi$ is zero. It follows that if the class of $\psi$ is nonzero then $\widetilde{\psi}$ can be identified with a linear form which is nonzero on $(\ker\widetilde{\varphi}_0)_\beta$.
\end{proof}\\
If $(\L_{n,\varphi_0}^{\mathrm{T},\mathcal{A}_n})_n$ is a direct filiation, then the construction of the scheme $\L_{n+1}^{\mathrm{T},\mathcal{A}_{n+1}}$ is obtained from the scheme $\L_{n}^{\mathrm{T},\mathcal{A}_{n}}$ by vanishing the new Jacobi's polynomials
\begin{equation}\label{e57}
	\J^{n+1}_{ijk}=\oint_{ijk}\sum_{l}Y^l_{ij}X^{n+1}_{lk},\quad 1<j<k\leq n,
\end{equation}
and the polynomial
\begin{equation}\label{e58}
	X^{n+1}_{pq}-1,
\end{equation}
where the variables $Y=(Y^l_{ij})$, $1<j\leq n,l\leq n$ are the ancient coordinates and $X^{n+1}_{lk}$ $(l,k\leq n)$ are the new. We can summary this study showing the different cases of the discussion of the extension $\varphi_n\rightarrow\varphi_{n+1}$. If $\O_n=\O_n(Y)$  is the local ring of $\L_{n}^{\mathrm{T},\mathcal{A}_{n}}$ at $\varphi_n$  then the local ring $\O_{n+1}=\O_{n+1}(X,Y)$ of $\L_{n+1}^{\mathrm{T},\mathcal{A}_{n+1}}$ at $\varphi_{n+1}$  is obtained by localizing the quotient ring $\mathrm{A}=\O_n\left[X^{n+1}_{ij}\right]/\J$
by the maximal ideal generated by the classes of representatives $f(X,Y)$ which vanish at $\varphi_{n+1}$,
where $\J$ is the ideal generated by the polynomials (\ref{e57}) and (\ref{e58}). We obtain the following theorem with $\nu_n:=\dim\mathrm{H}_2(\varphi_n)_{\alpha_{n+1}}$.
\begin{thm}\label{t6.1}Let $(\L_{n,\varphi_0}^{\mathrm{T},\mathcal{A}_n})_n$ be a direct filiation. Each point $\varphi_{n+1}$ of the fiber of $\varphi_n$ in $\L_{n+1}^{\mathrm{T},\mathcal{A}_{n+1}}(\K)$ verifies one of the following cases:
\begin{enumerate}
	\item If $\nu_n=0$ then the fiber does not exist.
	\item If $\nu_n=1$ then $\varphi_{n+1}$ is unique in the fiber and $\O_{n+1}$ is a quotient of $\O_n$. Moreover, if $\varphi_n$ is rigid then $\varphi_{n+1}$ is rigid.
	\item If $\nu_n>1$ then $\O_{n+1}$ is the localization ring at $(0,...,0)$ of the quotient of $\O_n[T_1,...,T_{\nu_n-1}]$ by the ideal generated by polynomials in $Y$ of the form
	$$c_0^{\lambda}(Y)+\sum_{0<k<\nu_n}c_k^{\lambda}(Y)T_k,\lambda\in\Lambda-\Lambda_1,\,\,c_l^{\lambda}(Y)\in\m(\O_n),\,0\leq l\leq \nu_n,$$
	indexed over the triples $\lambda=(i<j<k)$ satisfying $\alpha_i+\alpha_j+\alpha_k=\alpha_{n+1}$. The Krull's dimension $d$ of the scheme $\L_{n+1}^{\mathrm{T},\mathcal{A}_{n+1}}$ at $\varphi_{n+1}$ is minored by $\nu_n-1$. In particular $\varphi_{n+1}$ is not rigid.
\end{enumerate}
\end{thm}
\begin{proof} The case $\nu_n=0$ is trivial. If $\nu_n>0$, the Jacobi's polynomials indexed by the set $\Lambda$ of triples $i<j<k$ with $\alpha_i+\alpha_j+\alpha_k=\alpha_{n+1}$ and the condition $X^{n+1}_{pq}=1$ may be written as
$a_0^{\lambda}(Y)+\sum_{\mu\neq (p,q)}a_{\mu}^{\lambda}(Y)X_\mu,\quad\lambda\in\Lambda,$
where $\mu$ runs through the set $M$ of couples $i<j$ such that $\alpha_i+\alpha_j=\alpha_{n+1}$, $(i,j)\neq (p,q)$ and $a_0^{\lambda},a_\mu^{\lambda}\in\O_n$. Let $(A)$ be the system obtained by vanishing these polynomials. If we fix $Y$ to $\varphi_n$, the equation system $(A)$ becomes a equation system $(B)$ where the solutions in $X$ are formed of $\Psi\in{Z}^2(\varphi_n,\K e_{n+1})^{\mathrm{T}}$ such that $\Psi^{n+1}_{pq}=1$. It gives an affine space of dimension $\nu_n-1$, see Proposition \ref{nn52}. Let $\Lambda_1$ be a subset of $\Lambda$ such that the equation subsystem of $(B)$ indexed by $\Lambda_1$ is equivalent to $(B)$, and $\Lambda_1$ is minimal for this property. There is a subset $M_1\in M$ of cardinal $|\Lambda_1|$ such that the submatrix $(a^\lambda_\mu(\varphi_n))_{(\lambda,\mu)\in\Lambda_1\times M_1}$ is invertible. This matrix remains invertible if we take its values in $\O_n$ since its projection by $\mathrm{pr}:\O_n\rightarrow\O_n/\m(\O_n)$ is. It follows that the equation subsystem of $(A)$ indexed by $\Lambda_1$ permits to express the variables $(X^\mu)_{\mu\in M_1}$ in function of variables $(X^\rho)_{\rho\in M-M_1}$, i.e.
$
	X^\mu=A^\mu_0(Y)+\sum_\rho A^\mu_\rho(Y)X^\rho, \mu\in M_1,
$
where $A^\mu_0(Y),A^\mu_\rho(Y)\in\O_n$.
We substitute these expressions in the remaining expressions of $(A)$ for multi-indices $\Lambda-\Lambda_1$, which gives the new system of equations by vanishing the following expressions:
\begin{equation}\label{e6.12}
	b_0^\lambda(Y)+\sum_{\rho\in M-M_1} b^\lambda_\rho(Y)X^\rho\quad \lambda\in\Lambda-\Lambda_1.
\end{equation}
If we fix $Y$ to $\varphi_n$, then the equation system (\ref{e6.12}) corresponds to the system of equations $(B)$ indexed by $\Lambda-\Lambda_1$. Theses equations give anything more on $\K$, and are identically null, i.e $b^\lambda_0(Y)=b^\lambda_\rho(Y)=0$ for $\lambda\in\Lambda-\Lambda_1$ and $\rho\in M-M_1$. In other words, the elements $b^\lambda_0(Y)$ and $b^\lambda_\rho(Y)$ belong to the maximal ideal of $\O_n$. The ring $\mathrm{A}$ is equal to the quotient of $\O_n[X^\rho,\rho\in M-M_1]$ by the ideal generated by the terms (\ref{e6.12}). If we localize this ring at point $\varphi_{n+1}$, we will take new adapted variables $X^\rho-(\varphi_{n+1})^\rho$ which we write $T_k$ with indexation $k$ on $\left\{1,...,\nu_n-1\right\}$. Then the Jacobi's polynomials are written as $c_0^{\lambda}(Y)+\sum_{0<k<\nu_n}c_k^{\lambda}(Y)T_k,\quad\lambda\in\Lambda-\Lambda_1,$
with $c_l^{\lambda}(Y)\in\m(\O_n)$, $(0\leq l< \nu_n)$ and $\mathrm{A}$ is written as in $3.$ We localize at point $\varphi_{n+1}$ by considering the maximal ideal obtained as quotient of the maximal ideal of $\O_n[T]$ is $\m(\O_n)[T]+\sum_{1\leq k<\nu_n}T_k.\O_n[T]$.\\
If $\nu_n=1$, we see that the ring $\A$ is equal to the quotient of $\O_n$ by the ideal generated by the $b_0^\lambda$, $\lambda\in\Lambda-\Lambda_1$. We obtain the result.\\
If $\nu_n>1$, there are $(\nu_n-1)$ new parameters which are algebraically independent over $\K$, Then the Krull's dimension satisfies $d\geq\nu_n-1$.
\end{proof}
\subsection{Continuous families of Lie algebras}
\begin{de} \label{d5.2}A path of weights is said to be simple if all weights are distinct.
\end{de}
With this definition, the coordinates $X^k_{ij}$ may be indexed by the weights themselves. One can write $X_{ij}$ instead of $X^k_{ij}$ since the index $k$ is fixed by the weight $\alpha_i+\alpha_j$. The set of pairs $i<j$ such that $\alpha_i+\alpha_j\in\pi_n$ shall be denoted again by $\mathcal{J}$. Set $\G_n:=\mathrm{Gl}_n\left(\K\right)^{\mathrm{T}}_0$.
\begin{prop}Let $\mathrm{T}$ be a torus on $\K^n$. One supposes that its set of weights $\pi_n$ is a simple path of weights. Then $\sum_{n}\left(\mathrm{T}\right)$ is a Zariski's open set equal to the set of elements $\varphi_n\in\L_n^{\mathrm{T}}$ such that $\mathrm{T}\ltimes\varphi_n$ is complete. It is the union of the $\G_n$-orbits of maximal dimension $n-\dim_\K\mathrm{T}$.
\end{prop}
\begin{proof} If $\varphi$ is an element of $\L_n^{\mathrm{T}}(\K)$, then $\Der\left(\varphi\right)^\mathrm{T}$ is a torus $\tau$ containing $\mathrm{T}$. Then the law $\varphi$ satisfies $\mathrm{T}.\n+\left[\n,\n\right]=\n$ and $\n^{\mathrm{T}}\cap\mathrm{Z}(\n)=0$, and then $\mathrm{\tau}\ltimes\varphi$ is a complete Lie algebra, cf. \cite{C4}. Then the set $\sum_{n}\left(\mathrm{T}\right)$ consists of elements $\varphi$ such that $\tau=\mathrm{T}$, i.e. $\dim\tau=\dim\mathrm{T}$ since $\mathrm{T}\subset\tau$. It follows that it is the Zariski's open set union of orbits of maximal dimension.
\end{proof}
\begin{rem} Under the above hypotheses, any element $\varphi\in\L^{\mathrm{T}}_n$ belongs to the Zariski's open $\sum_{n}\left(\tau\right)$ of the subscheme $\L^{\tau}_n$ of $\L^{\mathrm{T}}_n$ in which we can apply the reduction theorem for $\tau\ltimes\varphi$. The different $\Sigma_n(\mathrm{\tau})$ form a stratification of $\L_n^{\mathrm{T}}$.
\end{rem}
The open stratum $\Sigma_n(\mathrm{T})$ gives a quotient variety $\Sigma_n(\mathrm{T})/\G_n$ under $\G_n$.
\begin{prop}\label{p5.1}The isomorphic classes in $\sum_n(\mathrm{T})$ are the orbits of the normalizer group $\mathrm{H}$ of $\mathrm{T}$ in $\Gl_n(\K)$ under the canonical action.
\end{prop}
\begin{proof} We can see that $\mathrm{H}$ stabilizes $\L_n^{\mathrm{T}}(\K)$ and $\sum_n(\mathrm{T})$. Conversely, if $\varphi_1$ and $\varphi_2$ are isomorphic, then there is $s\in\Gl_n(\K)$ such that $s\ast\varphi_1$ is equal to $\varphi_2$ with maximal torus $s\cdot\mathrm{T}\cdot s^{-1}$. The field $\K$ being algebraically closed, there is an automorphism $s'$ of $\varphi_2$ which conjugates $s\cdot\mathrm{T}\cdot s^{-1}$ and $\mathrm{T}$ according to a Mostow's theorem, we have $s'\cdot s\in\mathrm{H}$.
\end{proof}
\begin{prop}The isomorphic classes in the quotient $\Sigma_n(\mathrm{T})/\G_n$ are the orbits of the finite group $\Gamma:=\mathrm{H}/\mathrm{H}_0$, with $\mathrm{H}$ the normalizer group of $\mathrm{T}$ in $\Gl_n(\K)$ and $\mathrm{H}_0$ its identity component. This variety is called continuous family associated with the maximal torus $\mathrm{T}$.
\end{prop}
\begin{proof} It is a direct consequence of Proposition \ref{p5.1}.
\end{proof}\\
Next, we will study the quotient variety $\Sigma_n(\mathrm{T})/\G_n$ with the help of slices $\L^{\mathrm{T},\mathcal{A}}_{n,\varphi_0}$, $\varphi_0\in\Sigma_n(\mathrm{T})$, which are local affine charts.
The weights being distinct, then $\mathrm{G}_n$ is the diagonal group identified with $(\K^*)^n$. The canonical action of an element $s=\left(s_1,...,s_n\right)$ of $\mathrm{G}_n$ on $X$ is defined by $
	(s\ast X)^k_{ij}=\frac{s_k}{s_is_j}X^k_{ij},$
where $X$ is a law defined by its coordinates $X(e_i,e_j)=\sum_kX^k_{ij}e_k$. It is particularly easy to characterize an admissible set of $\mathcal{J}$ at a law $\varphi$. It is just a subset $\mathcal{A}$ of $\mathcal{J}$ such that the following system of equations 
\begin{equation}\label{E.6.2}
	\varphi^k_{ij}=\frac{s_k}{s_is_j}\varphi^k_{ij}\quad \mathrm{for}\quad ({i<j})\in\mathcal{A},
\end{equation}
is equivalent to the system $s\ast\varphi=\varphi,$
and is minimal for this property. The minimality of $\mathcal{A}$ implies that $\varphi_{ij}^k\neq 0$ for all $({i<j})\in\mathcal{A}$. It follows that $\mathcal{A}$ is contained in the set $\mathcal{J}_{\varphi}$ consisting of pairs $(i<j)$ such that $\varphi_{ij}^k\neq 0$.
\begin{prop}A subset $\mathcal{A}$ of $\mathcal{J}$ is admissible at $\varphi$ if and only if $\mathcal{A}$ is a minimal subset of $\mathcal{J}_{\varphi}$ such that the system
$
	s_k=s_is_j,(i<j)\in\mathcal{A},
$
defines the subgroup $\mathrm{Aut(\varphi)}_0^{\mathrm{T}}$ of $\mathrm{G}_n$.
\end{prop}  
The laws $\varphi$ of $\L_n^{\mathrm{T}}(\K)$ which have the same subset $\mathcal{J}_{\varphi}\subset\mathcal{J}$ thus have the same admissible sets $\mathcal{A}\subset\mathcal{J}_{\varphi}$, the same group $\mathrm{Aut(\varphi)}_0^{\mathrm{T}}$ and the same maximal torus $\tau$. 
Two equivalent laws under  $\mathrm{G}_n$ have the same subset $\mathcal{J}_{\varphi}$ of $\mathcal{J}$ since Eq (\ref{E.6.2}) implies that a nonzero coordinate remains nonzero under $\mathrm{G}_n$.
\begin{thm}\label{n51}Under the hypothesis of a simple path of weights, each law $\varphi\in\sum_{n}(\mathrm{T})$ admits an admissible set $\mathcal{A}$ of cardinal $|\mathcal{A}|=n-\dim_\K\mathrm{T}$. We have for such a subset $\mathcal{A}$:
\begin{enumerate}
	\item $\varphi$ admits $\mathcal{A}$ as admissible set iff $\varphi^k_{ij}\neq 0$ for all $(^k_{ij})\in\mathcal{A}$ where $i<j$.
	\item All laws of $\L_{n,\varphi_0}^{\mathrm{T},\mathcal{A}}(\K)$ admit $\mathcal{A}$ as admissible set.
	\item $\L_{n,\varphi_0}^{\mathrm{T},\mathcal{A}}(\K)$ is contained in $\sum_{n}\left(\mathrm{T}\right)$ and its isomorphic classes are the traces of the $\Gamma$-orbits.
	\item $\varphi$ admits $\mathcal{A}$ as admissible set iff there is an element $s\in\mathrm{G}_n$ such that $s\ast\varphi\in\L_{n,\varphi_0}^{\mathrm{T},\mathcal{A}}(\K)$.
	\item $\mathrm{H}^2(\varphi_0,\varphi_0)^{\mathrm{T}}$ is the Zariski's tangent space to $\L^{\mathrm{T},\mathcal{A}}_{n,\varphi_0}(\K)$ at $\varphi_0$.
\end{enumerate}
\end{thm}
\begin{proof} The laws $\varphi$ which admit $\mathcal{A}$ as admissible set satisfy $\varphi^\alpha\neq 0$ for all $\alpha\in\mathcal{A}$. This condition is sufficient if the cardinal of $\mathcal{A}$ is maximal. This proves the statements $1,2$ and $3$. 
The orbit $\left[\varphi_0\right]$ of $\varphi_0$ under $(\K^*)^n$ is given by
$$\left[\varphi_0\right]=\left\{(\frac{s_k}{s_is_j}(\varphi_0)_{ij}^k):s=(s_1,...,s_n)\in\mathrm{G}_n\right\}.$$
The components indexed by an admissible set $\mathcal{A}$ are non null and the corresponding terms $(\frac{s_k}{s_is_j})$ are free and can take arbitrary non null values; there thus  exits in the orbit of $\varphi_0$ a representative $\varphi$ for any arbitrary choice of non null components $(\varphi^\alpha)_{\alpha\in\mathcal{A}}.$ If $|\mathcal{A}|$ is maximal we obtain $4$, and $5$ results from transversal property of the slice.
\end{proof}\\
\textbf{Consequence.} Let $\Omega(\mathcal{A})$ be the set of the laws admitting $\mathcal{A}$ as admissible set. Then $\Omega(\mathcal{A})$ is an open set and $\sum_n(\mathrm{T})=\cup_{\mathcal{A}} \Omega(\mathcal{A})$. The slices associated with a same $\mathcal{A}$ are isomorphic and identified with the quotient $\Omega(\mathcal{A})/\G_n$. Varying the admissible parts $\mathcal{A}$, the slices are affine local charts of $\Sigma_n(\mathrm{T})/\G_n$ giving continuous families in Theorem \ref{n51} 3.\\ 
\textbf{Convention.} From Theorem \ref{n51}, 4. we may fix arbitrary values of the components $\varphi^{\mathcal{A}}$ of a law $\varphi$ in $\Sigma_n(\mathrm{T})$. One adopts the convention $X^k_{ij}=1$ for all $(^k_{ij})\in\mathcal{A}$. The corresponding subscheme, denoted by $\L_n^{\mathrm{T},\mathcal{A}}$, is isomorphic to $\L_{n,\varphi_0}^{\mathrm{T},\mathcal{A}}$.
\section{Examples}
All examples considered in this section satisfy the hypothesis of Proposition \ref{n3.3}, then the calculation of versal deformations of $\phi_m=\mathrm{T}\ltimes\varphi_n$ in $\L_m$ can be deduced from those of $\varphi_n$ in $\L^{\mathrm{T}}_n$. Also versal deformations of $\varphi_{n}\in\L_{n}^{\mathrm{T}}$ are calculated from that of $\varphi_{n_0}\in\L_{n_0}^{\mathrm{T}}$ by successive central extensions with $n_0$ the dimension where the maximal torus $\mathrm{T}$ appears.
\subsection{First Example}
Let $\mathfrak{f}_n$ be a Lie defined by $\varphi_0(e_1,e_i)=e_{1+i}$, $(i\leq n-1)$ and $\varphi_0(e_2,e_i)=e_{i+2}$ $(2<i<n-1)$. It admits a maximal torus $\mathrm{T}$ defined by its weights generated by $\alpha_i=i\alpha$, with $n\geq 5$. \\
The coordinates $X^k_{ij}$ of $\L_n^{\mathrm{T}}$ may be indexed by the weights themselves. One can write $X_{ij}$ instead of $X^k_{ij}$ and indexed by the pairs $(i<j)$ such that $\alpha_i+\alpha_j=\alpha_{i+j}$ is a weight. The Jacobi's polynomials are
\begin{equation}
	\mathrm{J}_{ijk}=X_{i,j}X_{i+j,k}+X_{j,k}X_{j+k,i}+X_{k,i}X_{k+i,j}
\end{equation}
for $i<j<k$ and $\alpha_i+\alpha_j+\alpha_k=\alpha_{i+j+k}\in\pi_n$.
Let $\K[t]_{(t)}$ denote the localization of $\K[t]$ at the prime $(t)$, which is isomorphic to  $\left\{\frac{p}{q}:p,q\in\K\left[u\right],q(0)\neq 0\right\}.$
\begin{thm} 
The set $\mathcal{A}_n=\left\{(23),(1k),1<k<n\right\}$ is admissible at $\varphi_0$.\\
The local ring $\O_{\varphi_0}^{\mathrm{T},\mathcal{A}_n}$ of the slice $\L^{\mathrm{T},\mathcal{A}_n}_{n,\varphi_0}$ at $\varphi_0$ and the versal deformation $X=(X_{ij})$ of $\varphi_0$ in $\L_n^{\mathrm{T}}$  associated with $\mathcal{A}_n$ are given by\\
$\textbf{1)}$ for $n=5,6$, $\O_{\varphi_0}^{\mathrm{T},\mathcal{A}_n}=\K$, $X=\varphi_0$,\\
$\textbf{2)}$ for $7\leq n\leq 11$, $\O^{\mathrm{T},\mathcal{A}_n}_{\varphi_0}=\K[t]_{(t)}$, $X$ is defined by (\ref{e6.25}) to (\ref{e6.43}),\\
$\textbf{3)}$ for $n\geq 12$, $\O_{\varphi_0}^{\mathrm{T},\mathcal{A}_n}=\K[u]/\left\langle u^5\right\rangle$,\\
$X_{ij}=1$, $(ij)\in\mathcal{A}_n$, $X_{24}=1$, $ X_{25}=1-t$, $X_{26}=1-2t$,\\
$X_{2,m-2}=1+(6-m)t+\frac{3m^2-45m+168}{4} t^2+\frac{-4m^3+105m^2-923m+2712}{8} t^3+\\
\frac{5m^4-192m^3+2812m^2-18579m+46608}{16}t^4$ for $n\geq m\geq 9$,\\
$X_{34}=X_{35}=t$, $X_{36}=t-\frac{3}{2}t^2+\frac{3}{4}t^3-\frac{3}{8}t^4$, 
$X_{37}=t-3t^2+\frac{3}{2}t^3-\frac{3}{4}t^4$,\\
$X_{3,m-3}=t+\frac{3}{2}(8-m)t^2+\frac{6m^2-111m+516}{4} t^3+\frac{-10m^3+303m^2-3110m+10794}{8} t^4$\\
for $n\geq m \geq 11,$
$X_{45}=X_{46}=\frac{3}{2}t^2-\frac{3}{4}t^3+\frac{3}{8}t^4,$ \\
$X_{4m-4}=\frac{3}{2}t^2+\frac{-12m+117}{4}t^3+\frac{30m^2-636m+3423}{8}t^4$, for $n\geq m\geq 11$\\
$X_{56}=X_{57}=3t^3-\frac{27}{4}t^4$\\
$X_{5m-5}=3t^3+\frac{-30m+333}{4}t^4$, for $n\geq m\geq 12$,\\
$X_{6m-6}=\frac{15}{2}t^4$ for $m\geq 13$ and $X_{ij}=0$, for $6<i<j$, with $t^5=0$.
\end{thm}
\begin{proof} The Lie algebra $\mathfrak{f}_n$ is built by successive unidimensional central extensions of nilpotent Lie algebras where the maximal torus $\mathrm{T}$ appears for $n=5$.\\
-For $n=6$ one obtains two new coordinates $X_{15},X_{24}$ and one relation $\J_{123}$ i.e.
\begin{equation}
	X_{24}=X_{15}=1.
\end{equation}
-For $n=7$ we have three new coordinates $X_{16},X_{25},X_{34}$ and one Jacobi's polynomial $\J_{124}$ i.e.
\begin{equation}\label{e6.25}
	X_{34}-1+X_{25}=0
\end{equation}
-For $n=8$ we have three new coordinates $X_{17},X_{26},X_{35}$ and two Jacobi's polynomials $\J_{125}$ and $\J_{134}$ i.e.
\begin{equation}\label{e6.31}
	X_{17}=1, X_{26}=1-2t, X_{35}=t.
\end{equation}
-For $n=9$ we have four new coordinates $X_{18},X_{27},X_{36},X_{45}$ and three Jacobi's polynomials $\J_{126},\J_{135},\J_{234}$, 
for $t\neq -2$ we have
\begin{equation}\label{e6.32}
	X_{18}=1, \ X_{27}=\frac{2-5t}{2+t}, \ X_{36}=\frac{2t-2t^2}{2+t}, \ X_{45}=\frac{3t^2}{2+t}.
\end{equation}
(\ref{e6.32}) is valid in the local ring at each point of $\L^{\mathrm{T},\mathcal{A}_9}_9(\K)$ since the projection of $2+t$ on $\K$ is different from zero and $2+t$ is invertible if $t\neq -2$.\\
For $n=10$ we have four new coordinates $X_{19},X_{28},X_{37},X_{46}$ and four Jacobi's polynomials $\J_{127},\J_{136},\J_{145},\J_{235}$, i.e.
\begin{equation}\label{e6.37}
	X_{19}=1, \ X_{28}=\frac{2-7t+5t^2}{2+t}, \ X_{37}=\frac{2t-5t^2}{2+t}, \ X_{46}=\frac{3t^2}{2+t}.
\end{equation}
in the associated local ring at each value $\mathrm{pr}(t)\in\K-\left\{-2\right\}$.\\
-For $n=11$ we have five new coordinates $X_{110},X_{29},X_{38},X_{47},X_{56}$ and five Jacobi's polynomials $\J_{128},\J_{137},\J_{146},\J_{236},\J_{245}$ which give
\begin{equation}	X_{110}=1, X_{29}=\frac{2-10t+16t^2-5t^3}{2(1-t^2)},X_{38}=\frac{4t-16t^2+8t^3-5t^4}{2(2+t)(1-t^2)},
\end{equation}
\begin{equation}\label{e6.43}
	X_{47}=\frac{6t^2-12t^3+15t^4}{2(2+t)(1-t^2)},X_{56}=\frac{12t^3-21t^4}{2(2+t)(1-t^2)}.
\end{equation}
in the associated local ring at each value $\mathrm{pr}(t)\in\K-\left\{-2,\pm 1\right\}$.\\
-For $n=12$ we have five new coordinates $X_{111},X_{210},X_{39},X_{48},X_{57}$ and seven Jacobi's polynomials $\J_{129},\J_{138},\J_{147},\J_{156},\J_{237},\J_{246},\J_{345}$, i.e.
\begin{align}\label{e6.51}
	X_{210}=\frac{4-22t+44t^2-26t^3+36t^4}{2(2+t)(1-t^2)},X_{39}=\frac{4t-22t^2+32t^3-41t^4}{2(2+t)(1-t^2)},\\X_{48}=\frac{6t^2-24t^3+36t^4}{2(2+t)(1-t^2)},X_{57}=\frac{12t^3-21t^4}{2(2+t)(1-t^2)}
\end{align}
\begin{equation}
	\J_{237}=-\J_{246}=\J_{345}=\frac{9t^5(10t-1)}{(2+t)^2(1-t)^2}=0.
\end{equation}
For $t=0$ (resp. $t=1/10$), the law corresponds to the Lie algebra $\mathfrak{f}_{12}$ (resp. Witt Lie algebra $\mathfrak{w}_{12}$), and then $\L_{12}^{\mathrm{T},\mathcal{A}_{12}}(\K)=\left\{\mathfrak{f}_{12},\mathfrak{w}_{12}\right\}$. We have $\O_{\mathfrak{f}_{12}}^{\mathrm{T},\mathcal{A}_{12}}\simeq\K[u]/\left\langle u^5\right\rangle$ with a nilpotent element $t=\overline{u}$ of order $5$, and $\O_{\mathfrak{w}_{12}}^{\mathrm{T},\mathcal{A}_{12}}\simeq\K$. It follows from Theorem \ref{t4.1} that $\mathfrak{f}_{12}$ and $\mathfrak{w}_{12}$ are formal rigid in $\L_{12}^{\mathrm{T}}$. Now we reason by induction on $n$. For $n+1\geq 13$, we separately study the central extensions of $\mathfrak{f}_n$ and of $\mathfrak{w}_n$ proceeding by induction. Let $Y_{ij}$ denote the ancient structure constants for $i+j\leq n$ and $X_{ij}$ the new variables for $i+j= n+1$. The relations $	\J_{1jn-j}=X_{j+1,n-j}-Y_{j,n-j}+X_{j,n+1-j}=0$
give
\begin{equation}\label{e6.56}
	X_{j,n+1-j}=\sum_{k=j}^{j'}(-1)^{k-j}Y_{k,n-k}+(-1)^{j'-j+1}X_{j'+1,n-j'}
\end{equation}
and it suffices to write another relation for fixing $X_{ij}$ in function of the $Y_{ij}$. In particular, the relation $\J_{24n-5}=0$ fixes $X_{6n-5}$.
For $\mathfrak{w}_n$, we obtain only the Witt Lie algebra $\mathfrak{w}_{n+1}$ and there are not new parameters; the others Jacobi's relations are automatically satisfied.\\
For $\mathfrak{f}_n$, the relation (\ref{e6.56}) and $\J_{24n-5}=0$ fix the structure constants at neighbourhood of $\mathfrak{f}_{n+1}$ which verify the given relations in the theorem. Hence there are not new parameter. We must verify that the remained Jacobi's relations do not reduce the order of nilpotency of the parameter $t$. It is obvious to see that the Jacobi's polynomials $\J_{ijk}$ $(i<j<k)$ are vanished for $j\geq 6$, and it remains to see that $\J_{ijk}$ for $2\leq i\leq j\leq 5$ are null too.
\end{proof}
\begin{cor}
\begin{enumerate}
	\item The slice $\L_{n,\varphi_0}^{\mathrm{T},\mathcal{A}_n}$ associated with $\mathcal{A}_n$ is isomorphic to \\$\mathrm{Spec}(\K[u]/\left\langle u^5(10u-1)\right\rangle)$, for $n\geq 12$, where $\overline{u}=\frac{1}{10}$ (resp. $\overline{u}=0$)  corresponds to the Witt Lie algebra $\mathfrak{w}_{n}$ (resp. $\mathfrak{f}_{n}$).
	\item The versal deformation of $\mathfrak{w}_n$ associated with $\mathcal{A}_n$ is constant, $\O_{\mathfrak{w}_n}^{\mathrm{T},\mathcal{A}_n}\simeq\K$.
\end{enumerate}
\end{cor}
\begin{rem}The Lie algebra $\mathfrak{f}_n$ is formal $\mathrm{T}$-rigid if $n=5,6$ and $n\geq 12$, and no formal $\mathrm{T}$-rigid for $9 \leq n \leq 11$. If we develop their versal deformations on the competion local ring, they will become Gerstenhaber's deformations.
\end{rem}
\subsection{Second Example}
Let $\mathrm{T}$ be a torus defined by the set $\pi$ of its weights $\alpha_i, 1\leq i\leq 4,\alpha_i+\alpha_j, (i,j)\neq (1,4),\alpha_i+\alpha_j+\alpha_k, (i,j,k)\neq (2,3,4), \delta:=\alpha_1+\alpha_2+\alpha_3+\alpha_4$.\\
The coordinates $X^k_{ij}$ of the scheme $\L_n^{\mathrm{T}}$ may be indexed by the weights themselves. One writes $X_{\alpha,\beta}$ instead of $X^k_{ij}$ since the index $k$ is fixed by the weight $\alpha+\beta$.
The Jacobi's polynomials are
\begin{equation}
	\mathrm{J}_{\alpha,\beta,\gamma}=X_{\alpha,\beta}X_{\alpha+\beta,\gamma}+X_{\beta,\gamma}X_{\beta+\gamma,\alpha}+X_{\gamma,\alpha}X_{\gamma+\alpha,\beta}
\end{equation}
for $\alpha,\beta,\gamma\in\pi$ and $\alpha+\beta+\gamma\in\pi$. 
We consider successive central extensions by the sums of $2,3$ and $4$ weights and $\mathcal{A}$ is an admissible set. \\
The law defined by the $\alpha_i$ for $1\leq i\leq 4$ is abelian and provides the initialization $n_0=4$.\\
Central extension by the sums of $2$ weights:  each $\alpha_i+\alpha_j$ with $(i,j)\neq (1,4)$ is associated to the coordinate $X_{\alpha_i,\alpha_j}$ without Jacobi's relation; set $\mathcal{A}_2=\left\{(\alpha_i,\alpha_j):(i,j)\neq(1,4)\right\}$ and fix  
\begin{equation}
	X_{\alpha,\beta}=1, \quad(\alpha,\beta)\in\mathcal{A}_2.
\end{equation}
Central extensions by the sums of $3$ weights: the weight $\alpha_1+\alpha_3+\alpha_4$ gives coordinates $X_{\alpha_1,\alpha_3+\alpha_4}$ and $X_{\alpha_4,\alpha_1+\alpha_3}$ (since $(i,j)\neq (14)$), and the Jacobi's polynomial $\mathrm{J}_{\alpha_1,\alpha_3,\alpha_4}$ imposes $X_{\alpha_4,\alpha_1+\alpha_3}=-X_{\alpha_1,\alpha_3+\alpha_4}$. In the same fashion, $\alpha_1+\alpha_2+\alpha_4$ gives $X_{\alpha_1,\alpha_2+\alpha_4}$ and $X_{\alpha_4,\alpha_1+\alpha_2}$, and $\mathrm{J}_{\alpha_1,\alpha_2,\alpha_4}$ implies that $X_{\alpha_4,\alpha_1+\alpha_2}=-X_{\alpha_1,\alpha_2+\alpha_4}$. For $\alpha_1+\alpha_2+\alpha_3$ we have $X_{\alpha_1,\alpha_2+\alpha_3}$, $X_{\alpha_2,\alpha_1+\alpha_3}$ and $X_{\alpha_3,\alpha_1+\alpha_2}$, and $\mathrm{J}_{\alpha_1,\alpha_2,\alpha_3}$ gives  $X_{\alpha_2,\alpha_1+\alpha_3}=X_{\alpha_1,\alpha_2+\alpha_3}+X_{\alpha_3,\alpha_1+\alpha_2}$. \\
Set $\mathcal{A}_3=\mathcal{A}_2\cup\left\{(\alpha_4,\alpha_1+\alpha_3),(\alpha_4,\alpha_1+\alpha_2),(\alpha_3,\alpha_1+\alpha_2)\right\}$ and fix $X_{\alpha,\beta}=1$ for all $(\alpha,\beta)\in\mathcal{A}_3-\mathcal{A}_2$. We have 
\begin{equation}	X_{\alpha_4,\alpha_1+\alpha_3}=-X_{\alpha_1,\alpha_3+\alpha_4}=X_{\alpha_4,\alpha_1+\alpha_2}=-X_{\alpha_1,\alpha_2+\alpha_4}=X_{\alpha_3,\alpha_1+\alpha_2}=1
\end{equation}
\begin{equation}
	X_{\alpha_2,\alpha_1+\alpha_3}=X_{\alpha_1,\alpha_2+\alpha_3}+1
\end{equation}
Central extensions by the sums of $4$ weights: $\delta$ gives $X_{\alpha_2,\alpha_1+\alpha_3+\alpha_4}$, $X_{\alpha_3,\alpha_1+\alpha_2+\alpha_4}$, $X_{\alpha_4,\alpha_1+\alpha_2+\alpha_3}$, $X_{\alpha_1+\alpha_2,\alpha_3+\alpha_4}$ and $X_{\alpha_1+\alpha_3,\alpha_2+\alpha_4}$. \\
Set $\mathcal{A}_4=\mathcal{A}_3\cup(\alpha_2,\alpha_1+\alpha_3+\alpha_4)$ and fix $X_{\alpha_2,\alpha_1+\alpha_3+\alpha_4}=1$. The Jacobi's relations $\mathrm{J}_{\alpha_i,\alpha_j,\delta-\alpha_i-\alpha_j}$ for $(i,j)\neq(2,3)$ give the following equations:
\begin{equation}
	X_{\alpha_1+\alpha_2,\alpha_3+\alpha_4}=X_{\alpha_2,\alpha_1+\alpha_3+\alpha_4}=1, X_{\alpha_1+\alpha_3,\alpha_2+\alpha_4}=X_{\alpha_3,\alpha_1+\alpha_2+\alpha_4} 
\end{equation}
\begin{equation}
	X_{\alpha_1+\alpha_3,\alpha_2+\alpha_4}=X_{\alpha_4,\alpha_1+\alpha_2+\alpha_3}-1
\end{equation}
\begin{equation}\label{E5.22}
	X_{\alpha_1,\alpha_2+\alpha_3}X_{\alpha_4,\alpha_1+\alpha_2+\alpha_3}=0
\end{equation}
With the choice of $\mathcal{A}:=\mathcal{A}_4$, we obtain two essentials parameters $u:=X_{\alpha_1,\alpha_2+\alpha_3}$ and $v:=X_{\alpha_4,\alpha_1+\alpha_2+\alpha_3}$ verifying (\ref{E5.22}), and we can state with $\varphi_0\in\L_{13}^{\mathrm{T},\mathcal{A}}$: 
\begin{thm}The slice $\L_{13}^{\mathrm{T},\mathcal{A}}$ is isomorphic to $\mathrm{Spec}(\K[u,v]/(uv))$.
\begin{enumerate}
	\item If $(\varphi_0)_{\alpha_1,\alpha_2+\alpha_3}=(\varphi_0)_{\alpha_4,\alpha_1+\alpha_2+\alpha_3}=0$, then 
	$\varphi_0$ is a singular point at the intersection of two components of the slice, and $\varphi_0+u\varphi_1+v\varphi_2$ with $uv=0$, is a versal deformation of $\varphi_0$ where $\varphi_1$ and $\varphi_2$ are two cocycles with non null classes. The local ring $\O_{\varphi_0}^{\mathrm{T},\mathcal{A}}$ is the localization of $\K[u,v]/\left\langle uv\right\rangle$ at the maximal ideal $\mathrm{P}$ consisting of $p$ such that $p(0,0)=0$, denoted by $\left(\K[u,v]/\left\langle uv\right\rangle\right)_{(\mathrm{P})}$.
	\item If $(\varphi_0)_{\alpha_1,\alpha_2+\alpha_3}\neq 0$ (resp. $(\varphi_0)_{\alpha_4,\alpha_1+\alpha_2+\alpha_3}\neq 0$), then $\varphi_0$ is a regular point at the one component of the slice, and $\varphi_0+u\varphi_1$ (resp.  $\varphi_0+v\varphi_2$) is a versal deformation of $\varphi_0$, where $\varphi_1$ and $\varphi_2$ are two cocycles with non null classes. The local ring $\O_{\varphi_0}^{\mathrm{T},\mathcal{A}}$ is  equal to $\left(\K[u,v]/\left\langle uv\right\rangle\right)_{(\mathrm{P})}$, where the maximal ideal $\mathrm{P}$ consists of $p$ such that $p(u_0,0)=0$, (resp . $p(0,v_0)=0$), and is isomorphic to $\K[\tau]_{(\tau)}$.
\end{enumerate}
\end{thm}
\begin{rem}We deduce from the reduction theorem that:
\begin{enumerate}
	\item [a)]$\phi_0+u\varphi_1+v\varphi_2$ is a versal deformation of $\phi_0=\mathrm{T}\ltimes\varphi_0$ in case 1 and $\phi_0+u\varphi_1$ or $\phi_0+v\varphi_2$ in case 2.
	\item [b)] the 2-cohomology group of $\phi_0$ is $\C^2$ in the case 1 and $\C$ in the case 2.
\end{enumerate}
\end{rem}
\subsection{Conclusion}
Gerstenhaber, Nijenhuis and Richardson did not study versal deformations in their work. However, they studied obstruction problems of one parameter deformations, and proved the link between rigidity of an algebra and its cohomology. 

To choose the local ring $\O$ as base is natural, and provides the canonical deformation $\mathrm{id}:\O\rightarrow\O$, which is a universal object in the category of deformations at $\phi_0$ satisfying the property (i) of Definition \ref{V}. However, this deformation is not minimal for this property, this is why we define a slice of $\L_m$ containing the point $\phi_0$ such that its canonical deformation in this subscheme is minimal, and called versal. The methods for computing versal deformations by fixing parameters and central extensions, and the reduction's theorem presented above simplify a lot of their calculations.

We can deduce all the deformation equivalence classes of $\varphi_0$ with base $\mathrm{A}\in\widehat{\got{R}}$ from its versal deformation associated with an admissible set $\mathcal{A}$, given by the local ring $\O_{\varphi_0}^{\mathrm{T},\mathcal{A}}$ of the slice at $\varphi_0$. Each deformation with base $\A$ is equivalent to $\widehat{\overline{h}}\circ\widehat{\pi}$, where the change base $\widehat{\overline{h}}:\widehat{\O_{\varphi_0}^{\mathrm{T},\mathcal{A}}}\rightarrow\A$ is defined by its image $\widehat{\overline{h}}(\widehat{\O_{\varphi_0}^{\mathrm{T},\mathcal{A}}})$ which corresponds to a quotient $\widehat{\O_{\varphi_0}^{\mathrm{T},\mathcal{A}}}/\mathrm{I}$ of $\widehat{\O_{\varphi_0}^{\mathrm{T},\mathcal{A}}}$, with $\mathrm{I}$ an ideal. To classify all deformation equivalence classes  of $\varphi_0$ with base $\A$ is equivalent to classify all ideals of $\mathrm{I}$ such that the local ring  $\widehat{\O_{\varphi_0}^{\mathrm{T},\mathcal{A}}}/\mathrm{I}$ is isomorphic to a local subring of $\A$.
In the case where $\A=\K[[\tau]]$ corresponds to the formal deformations in the Gerstenhaber's sense  of $\varphi_0$. Since $\K[[\tau]]$ is a domain, then the ideals $\mathrm{I}$ will be considered prime. 

In the first example, the point $\varphi_0$ is regular for $7\leq n\leq 11$, and the completion of the local  ring $\O_{\varphi_0}^{\mathrm{T},\mathcal{A}_n}=\K[t]_{(t)}$ is $\K[[t]]$. Then the versal deformation of $\varphi_0$ can be written as $\sum_nt^n\varphi_n$ on the completion ring, and gives one formal deformation of $\varphi_0$ (up an equivalence).\\
For $n\geq 12$, the Krull's dimension of the completion of the local ring  $\O_{\varphi_0}^{\mathrm{T},\mathcal{A}_n}\simeq\K[u]/\left\langle u^5\right\rangle$ is null, then $\varphi_0$ is formal rigid, cf. Theorem \ref{t4.1}, and any formal deformation of $\varphi_0$ is trivial.

In the second example, the point $\varphi_0$ at the intersection of two lines admits a versal deformation $\varphi_0+u\varphi_1+v\varphi_2$. The completion of the local ring $\O_{\varphi_0}^{\mathrm{T},\mathcal{A}}=\left(\K[u,v]/\left\langle uv\right\rangle\right)_{(\mathrm{P})}$ is $\K[[u,v]]/\left\langle uv\right\rangle$. Let $\overline{u}$ (resp. $\overline{v}$) be denote the class of $u$ (resp. $v$) in $\widehat{\O_{\varphi_0}^{\mathrm{T},\mathcal{A}}}$. It is clear that the ideals $\mathrm{I}_1=\left\langle\overline{ u}\right\rangle$ and $\mathrm{I}_2=\left\langle \overline{v}\right\rangle$ are prime. The Krull's dimension of $\widehat{\O_{\varphi_0}^{\mathrm{T},\mathcal{A}}}$ being one, we have only
three prime ideals $\mathrm{I}_1$, $\mathrm{I}_2$ and $\mathrm{P}$ where $\mathrm{P}$ is the maximal ideal consisting of $p\in\widehat{\O_{\varphi_0}^{\mathrm{T},\mathcal{A}}}$ such that $p(0,0)=0$. This gives three formal
deformations $\psi_i:=\varphi_0+\tau_i\varphi_i$ for $\widehat{\overline{h}_i}:\widehat{\O_{\varphi_0}^{\mathrm{T},\mathcal{A}}}\rightarrow\widehat{\O_{\varphi_0}^{\mathrm{T},\mathcal{A}}}/\mathrm{I}_i\stackrel{\simeq}{\rightarrow}\K[[\tau_i]]$ for $=1,2$ and $\varphi_0$ for $\widehat{\overline{h}}:\widehat{\O_{\varphi_0}^{\mathrm{T},\mathcal{A}}}\rightarrow\widehat{\O_{\varphi_0}^{\mathrm{T},\mathcal{A}}}/\mathrm{P}\stackrel{\simeq}{\rightarrow}\K$. The deformations $\psi_1$ and $\psi_2$ are not equivalent because $\varphi_1$ and $\varphi_2$ are linearly independent in $\mathrm{H}^2(\n,\n)^{\mathrm{T}}$. Consequently, there are two nontrivial formal deformations $\psi_1$ and $\psi_2$ of $\varphi_0$ (up an equivalence).


\begin{thebibliography}{9}
\bibitem{A} M.~Artin, On the solutions of analytic equations, Inv. Math. 5, 277-291 (1968).
\bibitem{B} H.~Bjar, O.A.~Laudal, Deformation of a Lie algebras and Lie algebras of deformations. Compositio Math. 75 (1900), no. 1, 76-111
\bibitem{C1} R.~Carles, Sur la structure des alg\`ebres de Lie rigides, Ann. Inst. Fourier, Grenoble, vol. 34,3 (1984), pp. 65-82.
\bibitem{C2} R.~Carles, Sur certaines classes d'alg\`ebres de Lie rigides, Math ann, 272, 477-488, (1985).
\bibitem{C7} R.~Carles, Introduction aux d\'eformations d'alg\`ebres de Lie de dimension finie, Preprint Universit\'{e} de Poitiers, n°19, 1986.
\bibitem{C3} R.~Carles, D\'eformations et \'el\'ements nilpotents dans les sch\'emas d\'efinis par les identit\'es de Jacobi. CRAS,312,671, 1991.
\bibitem{C4} R.~Carles, Construction des alg\`ebres de Lie compl\`etes, CRAS, 318, 711-714 (1994).
\bibitem{C9} R.~Carles, D\'eformations dans les sch\'emas d\'efinis par les identit\'es de Jacobi. Ann. Math. Blaise Pascal 3 (1996), no. 2, 33--62.
\bibitem{FF}A. Fialowski, B.D. Fuchs, Construction of Miniversal Deformations of Lie Algebras. Journal of Functional Analysis, 161/1 (1999), pp. 76-110.
\bibitem{G} M.~Gerstenhaber, On the deformations of rings and algebras II. Ann. of Math., vol. 79 (1964), 59-103.
\bibitem{HS} G.~Hochschild, J-P.~Serre, Cohomology of Lie algebras, Ann.~Math. 57 (1953), 72-144.
\bibitem{Ko} M. Kontsevich, Topics in Algebra: Deformation Theory, Lectures Notes, Univ. Calif. Berkeley (1994).
\bibitem{NR} A.~Nijenhuis and R.W.Richardson,Cohomology and deformations in graded Lie algebras, Bull. Amer. Math. Sco., vol 72 (1966), pp. 1-29.
\bibitem{R} G.~Rauch, Effacement et d\'eformation. Ann.Ins.Fourier, vol. 22 (1972), 239-269.
\bibitem{S} M.~Schlessinger, Functors of Artin rings, Trans. Amer. Math. Sco. vol. 130 (1968), pp. 208-222.
\end{thebibliography}
\end{document}